\documentclass[12pt,a4paper]{amsart}

\usepackage{graphicx}
\usepackage[inline]{enumitem}
\usepackage{amsmath, amssymb, amsthm, amscd}
\usepackage{xcolor}
\usepackage{caption}
\usepackage{subcaption}
\usepackage{multirow}
\usepackage{stmaryrd}
\usepackage{accents}
\usepackage{bm}
\usepackage{orcidlink}

\theoremstyle{plain}
\newtheorem{theorem}{Theorem}[section]
\newtheorem{lemma}[theorem]{Lemma}

\newtheorem{proposition}[theorem]{Proposition}

\theoremstyle{definition}
\newtheorem{definition}[theorem]{Definition}

\theoremstyle{remark}
\newtheorem{remark}[theorem]{Remark}
\numberwithin{equation}{section}

\newcommand{\leqs}{\leqslant}
\newcommand{\uu}[1]{\bm{#1}}
\newcommand{\abs}[1]{\lvert#1\rvert}
\newcommand{\D}{\partial}
\newcommand{\dd}{\mathrm{d}}
\newcommand{\dive}{\mathrm{div}}
\newcommand{\bdive}{\uu{\mathrm{div}}}

\newcommand{\Dt}{\partial_t}

\newcommand{\Dei}{\D_\mcal{E}^{(i)}}

\newcommand{\grd}{\nabla}
\newcommand{\bgrd}{\uu{\nabla}}

\newcommand{\mbb}{\mathbb}
\newcommand{\mbf}{\mathbf}
\newcommand{\mcal}{\mathcal}
\newcommand{\CFL}{\mathrm{CFL}}
\newcommand{\const}{\mathrm{const}}

\newcommand{\norm}[1]{\lVert#1\rVert}
\newcommand{\Norm}[1]{{\left\vert\kern-0.25ex\left\vert\kern-0.25ex\left\vert #1 
    \right\vert\kern-0.25ex\right\vert\kern-0.25ex\right\vert}}

\newcommand{\half}{\frac{1}{2}}

\newcommand{\e}{\epsilon}
\newcommand{\s}{\sigma}

\newcommand{\M}{\mcal{M}}
\newcommand{\E}{\mcal{E}}

\newcommand{\Ds}{D_\sigma}

\newcommand{\Eint}{\mcal{E}_{\mathrm{int}}}
\newcommand{\Ek}{\mcal{E}(K)}
\newcommand{\Lm}{L_{\M}(\Omega)}

\newcommand{\Hez}{\uu{H}_{\E,0}(\Omega)}
\newcommand{\dt}{\delta t}

\newcommand{\fesig}{F_{\epsilon , \sigma}}

\newcommand{\intr}{\mathrm{int}}
\newcommand{\extr}{\mathrm{ext}}

\newcommand{\absq}[1]{\abs{#1}^2}
\newcommand{\diam}{\text{diam}}

\def\div{\,\mbf{div}}

\newcommand\ubar[1]{%
\underaccent{\bar}{#1}}

\title[An Energy Stable and Well-balanced Scheme]{An Energy Stable and
  Well-balanced Scheme for the Ripa System}  

\author[Arun]{K.\ R.\ Arun\scalebox{1.2}{\orcidlink{0000-0001-9676-861X}}}
\address{School of Mathematics, Indian Institute of Science Education
  and Research Thiruvananthapuram, Thiruvananthapuram 695551, India}  
\email{arun@iisertvm.ac.in, rahuldev19@iisertvm.ac.in}
\thanks{K.R.A.\ acknowledges the support from Science and Engineering
  Research Board, Department of Science \& Technology, Government of
  India through grant CRG/2021/004078.} 

\author[Ghorai]{R.\ Ghorai\,\scalebox{1.2}{\orcidlink{0000-0003-0549-3417}}}
\thanks{R.G.\ would like to thank the Ministry of Education,
  Government of India for the PMRF fellowship support.} 

\date{\today}



\begin{document}

\maketitle

\begin{abstract}
  We design and analyse an energy stable, structure preserving and well-balanced scheme for the Ripa system of shallow water equations. The energy stability of the numerical solutions is achieved by introducing appropriate stabilisation terms in the discretisation of the convective fluxes of mass and momenta, the pressure gradient and the topography source term. A diligent choice of the interface values of the water height and the temperature ensures the well-balancing property of the scheme for three physically relevant hydrostatic steady states. The explicit in time and finite volume in space scheme preserves the positivity of the water height and the temperature, and it is weakly consistent with the continuous model equations in the sense of Lax-Wendroff. The results of extensive numerical case studies on benchmark test problems are presented to confirm the theoretical findings. 
\end{abstract}

\section{Introduction}
\label{sec:intro}

We consider the Ripa system of shallow water equations which takes temperature fluctuations into account. The Ripa system models ocean currents and it can be obtained by depth averaging multilayer ocean models \cite{Rip93, Rip95}. The governing equations in conservative form read   
\begin{subequations}
\label{eq:ripa}
    \begin{align}
   \Dt h + \dive(h \uu{u}) &= 0, \label{eq:ht_balance} \\
    \Dt(h \uu{u}) + \bdive (h \uu{u} \otimes \uu{u}) + \bgrd
      \big(\half g h^2\theta\big) &= -g h\theta\bgrd
                                    b, \label{eq:mom_balance}\\ 
    \Dt(h \theta) + \dive(h\theta\uu{u}) &= 0. \label{eq:temp_balance}
\end{align}
\end{subequations}
We take the independent variables $(t,\uu{x})\in
Q_T:=(0,T)\times\Omega$, where $\Omega\subset\mbb{R}^d,\,d=1,2$, is
bounded and open. The unknown $h=h(t,\uu{x})>0$ represents the water
depth, $\uu{u}=\uu{u}(t,\uu{x})\in\mbb{R}^d$, the velocity field and 
$\theta=\theta(t,\uu{x})>0$, the potential temperature. The constant
$g>0$ denotes the gravitational acceleration, and $b=b(\uu{x})$ is the
known bottom topography function. As long as $h\theta>0$, the Ripa
system \eqref{eq:ripa} remains hyperbolic. Since the solutions of
hyperbolic systems are known to develop discontinuities in finite
times, an entropy function is required to single out the meaningful
weak solutions. In the case of the Ripa system \eqref{eq:ripa}, the
role of the entropy is played by the total energy which satisfies the
identity  
\begin{equation}
\label{eq:cont_engy_cons}
    \Dt\Big(\half g h^2\theta+g h\theta b+\half h\abs{\uu{u}}^2\Big) +
    \dive \Big(\big(g h^2\theta +g h\theta b + \half
    h\abs{\uu{u}}^2\big)\uu{u}\Big)=0, 
\end{equation}
for smooth solutions. 

A crucial requirement while numerically simulating hyperbolic balance
laws is to preserve stationary solutions of interest. For the Ripa system
\eqref{eq:ripa}, a physically relevant stationary solution is the
hydrostatic steady state wherein the velocities vanish and the
pressure gradient exactly balances the bottom topography term. Due to
the presence of $\theta$ in the pressure term as well as in the source 
term, characterising even the hydrostatic steady states of
\eqref{eq:ripa} is difficult, unlike the classical shallow water equations. However, by imposing more restrictions,
the following three pertinent hydrostatic solutions can be obtained. \\
\begin{subequations}
\label{eq:steady_states}
Lake at rest steady state:
\begin{equation}
    \label{eq:lake_rest_ss}
        \uu{u}=0, \ \theta=\const, \ h+b=\const.
\end{equation}
Isobaric steady state:
\begin{equation}
    \label{eq:isobaric_ss}
        \uu{u}=0, \ b=\const, \ h^2\theta=\const.
\end{equation}
Constant water height steady state:
\begin{equation}
    \label{eq:const_ht_ss}
        \uu{u}=0, \ h=\const, \ b+\frac{h}{2}\ln\theta=\const.
\end{equation}
\end{subequations}
Unfortunately, classical numerical schemes often fail to preserve such steady states for large times within an acceptable range of accuracy. In the literature, a cure to this ailment is provided via the so-called `well-balanced' schemes. A well-balanced scheme exactly satisfy discrete counterparts of the steady states of the corresponding continuum equations. The well-balancing methodology has since been adopted and developed for formulating robust numerical schemes for several hydrodynamic models with source terms; see, e.g.\ \cite{ABB+04, Bou04, FMT11, Gos00, KP07, NXS07, XS06, Xu02} for a few references. Though the Ripa model \eqref{eq:ripa} is an extension of the classical shallow water system, straightforward augmentations of the well-balancing methodologies of shallow water equations to the Ripa system are not common. Further, such extensions need not preserve the hydrostatic steady states of the Ripa system. Nonetheless, in the literature, one can find different approaches to derive well-balanced schemes for the Ripa system \cite{BX20, DZB+16, HL16, SMC16, TK15}.  

The aim of the present work is to design and analyse an explicit in time, well-balanced, energy stable and structure preserving finite volume scheme for the Ripa system \eqref{eq:ripa}. An invariant domain preserving or structure preserving numerical scheme is a one that respect the important physical stability properties of the corresponding continuum equations, such as preserving the positivity of mass density and temperature, satisfying the energy stability and so on. Our careful design to achieve these salient features involves two key ingredients. The first is introducing a shift in the velocities, the pressure gradient and the source term, which induces the dissipation of mechanical energy; see also \cite{AGK23, AGK24, DVB17, DVB20, PV16}. The second is a careful choice of the water height and the temperature on the interfaces while discretising the fluxes and the bottom topography term. This particular interface discretisation helps in emulating the energy balance at the discrete level and subsequently leads to the well-balancing property. The energy stability further ensures the well-balancing property in the sense that the numerical solutions generated from initial conditions which are perturbations of the steady state remain numerically stable, i.e.\ no spurious solutions are generated. In addition to being energy stable and well-balanced, the present scheme renders the positivity of water height and temperature under a CFL-like timestep restriction. Further, the scheme is weakly consistent with the continuous equations in the sense of Lax-Wendroff, as the mesh parameters go to zero. The results of numerical case studies performed on benchmark test problems confirm the theoretical findings.

The rest of this paper is organised as follows. In Section~\ref{sec:space_discr}, we define the function spaces for the unknown functions, and construct the numerical fluxes and the discrete differential operators. Section~\ref{sec:fv_schm} is devoted to proving the energy stability, the structure preserving and the well-balancing properties of the scheme. The Lax-Wendroff consistency of the scheme is shown in Section~\ref{sec:cons_LW}. We corroborate the proposed claims about the scheme with several numerical case studies in Section~\ref{sec:numerics}. Lastly, we close the paper with some concluding remarks in Section~\ref{sec:conclusion}.

\section{Space Discretisation and Discrete Differential Operators} 
\label{sec:space_discr}
This section aims to define the relevant function spaces to
approximate the Ripa system \eqref{eq:ripa} within a finite volume
framework by discretising the computational space-domain $\Omega$
using a Marker and Cell (MAC) grid. 
\subsection{Mesh and unknowns}
\label{subsec:mesh}
We begin by defining a primal mesh $\M$ which is obtained by splitting
the domain $\Omega$ into a finite family of disjoint, possibly
non-uniform, closed rectangles $(d = 2)$ or parallelepipeds $(d = 3)$,
denoted by $K$, called primal cells. We then denote respectively by
$\Ek$ and $\E$, the set of all edges of a primal cell $K\in\M$ and the
collection of all edges of all cells in $\M$. Furthermore, we decompose
$\E$ as $\E = \bigcup_{i=1}^{d} \E^{(i)}$, where $\E^{(i)} =
\E^{(i)}_{\intr} \cup \E^{(i)}_{\extr}$. Here, $\E^{(i)}_{\intr}$ and
$\E^{(i)}_{\extr}$ denote, respectively, the collection of internal
and external edges of dimension $d - 1$ that are orthogonal to the
$i$-th unit vector $\uu{e}^{(i)}$ in the canonical basis of
$\mbb{R}^d$. Any internal edge $\s\in\Eint$ is denoted by $\s=K|L$, if
there exists $(K,L)\in\M^2$ with $K\neq L$ such that
$\bar{K}\cap\bar{L}=\s$. We are now in a position to define $\Ds$, a
dual cell corresponding to $\s$, as $\Ds=D_{K,\s}\cup D_{L,\s}$, if
$\s=K|L\in\Eint$ and $\Ds=D_{K,\s}$, if $\s\in\E_\extr\cap\Ek$. In the
above $D_{K,\s}$ and $D_{L,\s}$ are, respectively, the portions of the
primal cells $K$ and $L$ which constitute the dual cell $D_\s$.  

By $L_\M(\Omega)\subset L^{\infty}(\Omega)$ and
$H^{(i)}_{\E}(\Omega)$, we denote the space of all scalar-valued
functions that are constant on each primal cell $K\in\M$ and the
space of all vector-valued functions that are constant on each dual
cell $\Ds$ for $\s\in\E^{(i)}$, respectively. The space of
vector-valued functions vanishing on the external edges is denoted and
defined by $\uu{H}_{\mcal{E},0}(\Omega)=\prod_{i=1}^d
H^{(i)}_{\mcal{E},0}(\Omega)$, where  $H^{(i)}_{\mcal{E},0}(\Omega)$
contains those elements of $H^{(i)}_{\mcal{E}}(\Omega)$ which vanish
on the external edges. The dual cell average of a grid function $q\in
L_\M(\Omega)$ is defined by  
\begin{equation}
    \label{eq:dual_avg}
    q_{D_\s}=\frac{1}{\abs{D_\s}}\big(\abs{D_{K,\s}}q_K+\abs{D_{L,\s}}q_L\big), \ \s=K|L.
\end{equation}
We refer the reader to, e.g.\ \cite{AGK23,GHL+18,GHM+16} for a
detailed description of the notations used throughout this paper. 

\subsection{Discrete convective fluxes and discrete differential operators}
\label{subsec:fluxes_opr}

\begin{definition}[Discrete mass flux]
  For each $K\in\M$ and $\s\in\E(K)$, the discrete mass flux $F_{\s,K}
  \colon L_{\M}(\Omega) \times \uu{H}_{\E,0}(\Omega) \to \mbb{R}$ is
  defined by 
  \begin{equation}
    \label{eq:mass_flux}
    F_{\sigma,K}(h,\uu{v}):=\abs{\sigma}h_{\s}v_{\sigma, K}, \
    (h,\uu{v})\in L_{\mcal{M}}(\Omega) \times
    \uu{H}_{\mcal{E},0}(\Omega), 
  \end{equation}
  where $v_{\s,K}= v_\s\uu{e}^{(i)}\cdot\uu{\nu}_{\s,K}$, with
  $\uu{\nu}_{\s,K}$ denoting the unit normal on the interface
  $\s\in\E(K)$ in the outward direction of $K$. The value of the water
  height at the interface $\s$, denoted by $h_\s$, is approximated by
  the following centred or upwind choices:  
  \begin{equation}
    \label{eq:ht_interf}
    h_\s=
    \begin{cases}
      h_{\s,c}, & \mbox{ for centred scheme},\\
      h_{\s,up}, & \mbox{ for upwind scheme}.
    \end{cases}
  \end{equation}
  For any $h\in L_\M(\Omega)$ and $\s=K|L$, we define $h_{\s,c}$
  and $h_{\s,up}$ as follows 
  \begin{equation}
    \label{eq:ht_cent_upw}
    h_{\s,c}=\half(h_K+h_L) \ \mbox{and} \ 
    h_{\s,up}=
    \begin{cases}
      h_K, & \mbox{if} \ v_{\s,K}\geqslant 0,\\
      h_L, & \mbox{otherwise}.
    \end{cases}
  \end{equation}
\end{definition}
\begin{definition}[Discrete temperature flux]
  For each $K\in\M$ and $\s\in\E(K)$, the discrete temperature flux
  $G_{\s,K} \colon L_{\M}(\Omega) \times \uu{H}_{\E,0}(\Omega)\times
  L_{\M}(\Omega) \to \mbb{R}$ is defined by 
  \begin{equation}
    \label{eq:temp_flux}
    G_{\s,K}(h,\uu{v},\theta):=\abs{\s}(h\theta)_\s v_{\s,K}.
  \end{equation}
  The interface approximation $(h\theta)_\s$ is given by
  \begin{equation}
    \label{eq:tht_interf}
    (h\theta)_\s=
    \begin{cases}
      (h\theta)_{\s,c}, & \mbox{ for the centred scheme},\\
      (h\theta)_{\s,up}, & \mbox{ for the upwind scheme},
    \end{cases}
  \end{equation}
  where
  \begin{equation}
    \label{eq:temp_cent}
    (h\theta)_{\s,c}=
    \begin{cases}
      h_K\theta_\s, & \mathrm{if} \ h_K=h_L, \\
      \half(h_K\theta_K+h_L\theta_L), & \mathrm{otherwise}, 
    \end{cases}
    \quad \text{and} \quad
    (h\theta)_{\s,up}=
    \begin{cases}
      h_K\theta_\s, & \mathrm{if} \ h_K=h_L, \\
      h_{\s,c}\theta_K, & \mathrm{if} \ \theta_K=\theta_L, \\
      h_{\s,up}\theta_{\s,up}, & \mathrm{otherwise},
    \end{cases}
  \end{equation}
  for $\s=K|L\in\Eint$. The interface choice for $\theta$ is given by
  the following logarithmic average:  
  \begin{equation}
    \label{eq:theta_if}
    \theta_\s =
    \begin{cases}
      \frac{\theta_L-\theta_K}{\ln\theta_L-\ln\theta_K},\;\mathrm{if}\;\theta_K\neq
      \theta_L,   \\
      \theta_K,\;\mathrm{if}\;\theta_K=\theta_L.
    \end{cases}
  \end{equation}
\end{definition}
\begin{definition}[Discrete momentum flux]
    For a fixed $i=1,2,\dots,d$, for each $\s\in\E^{(i)},
    \epsilon\in\tilde{\E}(\Ds)$ and $(h,\uu{v},u)\in \Lm\times
    \uu{H}_{\E,0}\times H^{(i)}_{\E,0}$, the discrete upwind momentum
    convection flux is defined as 
    \begin{align} 
    \label{mom_flux_up} 
      \sum_{\epsilon\in\tilde{\E}(\Ds)}\fesig(h,\uu{v})u_{\epsilon, \mathrm{up}},
    \end{align}
    where $\fesig(h,\uu{v})$ is the mass flux across the edge
    $\epsilon$ of the dual cell $\Ds$, which is a suitable linear
    combination of the primal mass convection fluxes
    \cite{GHL+18}. The upwind velocity $u_{\epsilon, \mathrm{up}}$ is
    defined by  
    \begin{equation}
    \label{eq:mom_up}
    u_{\epsilon,\mathrm{up}}=
    \begin{cases}
      u_{\s}, &\mbox{if} \ \fesig(h,\uu{v})\geqslant 0,\\
      u_{\s^{\prime}}, &\mathrm{otherwise},
    \end{cases}
    \end{equation}
    with $\epsilon\in\tilde{\E}(\Ds)$, $\epsilon=\Ds|D_{\s^{\prime}}$.
\end{definition}

\begin{definition}[Discrete gradient and discrete divergence]
  The discrete gradient operator
  $\grd_{\E}:L_{\mcal{M}}(\Omega)\rightarrow\uu{H}_{\E}(\Omega)$ is
  defined by the map $q \mapsto
  \grd_{\mcal{E}}q=\Big(\D^{(1)}_{\E}q,\D^{(2)}_{\mcal{E}}q,\dots,\D^{(d)}_{\mcal{E}}q\Big)$,
  where for each $i=1,2,\dots,d$, $\partial^{(i)}_{\E}q$ denotes 
  \begin{equation}
    \label{eq:dis_grad}
    \partial^{(i)}_{\E}q=\sum_{\s\in
      \mcal{E}^{(i)}_\intr}(\partial^{(i)}_{\E}q)_{\s}\mcal{X}_{D_{\s}},
    \ \mbox{with} \  (\partial^{(i)}_{\mcal{E}}q)_{\s}=
    \frac{\abs{\s}}{\abs{D_\s}}(q_{L}-q_K)\uu{e}^{(i)}\cdot
    \uu{\nu}_{\s,K}, \; \s=K|L\in\E_\intr^{(i)}. 
  \end{equation}
  The discrete divergence operator
  $\dive_\M:\uu{H}_{\E}(\Omega)\rightarrow L_{\mcal{M}}(\Omega)$ is
  defined as $\uu{v} \mapsto \dive_\M
  \uu{v}=\sum_{K\in\M}(\dive_{\mcal{M}} \uu{v})_K \mcal{X}_{K}$, where
  for each $K\in\M$, $(\dive_{\mcal{M}} \uu{v})_K $ denotes 
  \begin{equation}
    \label{eq:dis_div}
    (\dive_\M \uu{v})_K
    =\frac{1}{\abs{K}}\sum_{\sigma\in\mcal{E}(K)}\abs{\sigma}
    v_{\sigma,K}, 
  \end{equation}
  where $v_{\sigma,K}$ is as in the mass flux, cf.\ \eqref{eq:mass_flux}.
\end{definition}

We have assumed that the discrete gradient vanishes on the boundaries
for the ease of analysis. The discrete gradient and divergence
operators defined as above are dual with respect to the $L^2$ inner
product, i.e.\  
\begin{equation}
    \label{eq:dis_dual}
    \int_{\Omega}q(\dive_\M
    \uu{v})\dd\uu{x}+\int_{\Omega}\grd_{\mcal{E}}q\cdot\uu{v}\dd\uu{x}=0, 
\end{equation}
for any $(q,\uu{v})\in\Lm\times\Hez$; see \cite[Lemma 2.4]{GHL+18} for
details. 

\section{The Finite Volume Scheme}
\label{sec:fv_schm}
In this section, we present a finite volume scheme for the Ripa model
\eqref{eq:ripa} based on the MAC discretisation and the discrete
differential operators as defined in
Section~\ref{sec:space_discr}. The main highlight of the present
scheme is the nonlinear energy stability of the solutions and its
well-balancing property for hydrostatic equilibria. To this end, we
adopt the formalism introduced in \cite{DVB17, DVB20, PV16} and later 
pursued in \cite{AGK23,AGK24,AK24}.

\subsection{An explicit numerical scheme}
\label{subsec:num_schm}
Let $(t^n)_{0\leqs n\leqs N}$ be a partition of the time interval
$[0,T]$ such that $0=t^0 < t^1 < \dots < t^N=T$ and let the timestep
$\dt=t^{n+1}-t^n$ be constant for the sake of simplicity. We take 
initial approximations for $h$ and $\theta$ as the averages of the
initial conditions $h_{0}$ and $\theta_0$ on the primal
cells. Analogously, we take initial approximation for $\uu{u}$ as the
average of the initial data $\uu{u}_{0}$ on the dual cells, i.e.\ 
\begin{subequations}
\label{eq:dis_ic}
\begin{gather}
    h_{K}^{0}=\frac{1}{|K|}\int_{K}h_{0}(\uu{x})\dd\uu{x}, \; \forall K\in\M,\\
    \theta_{K}^{0}=\frac{1}{|K|}\int_{K}\theta_{0}(\uu{x})\dd\uu{x}, \; \forall K\in\M,\\
    u_{\s}^{0}=\frac{1}{|\Ds|}\int_{\Ds}(\uu{u}_{0}(\uu{x}))_{i}\dd\uu{x}, \; \forall \s\in\E_{\intr}^{(i)}, \, 1\leqs i\leqs d.  
\end{gather}
\end{subequations}
An approximate solution $(h^{n+1},\uu{u}^{n+1},\theta^{n+1})\in
L_{\mcal{M}}(\Omega)\times \uu{H}_{\mcal{E},0}(\Omega)\times
L_\M(\Omega)$ of \eqref{eq:ripa} at time $t^{n+1}$ is computed
successively by the following explicit scheme: 
 \begin{subequations}
\label{eq:dis_updt}
\begin{gather}
  \frac{1}{\delta
    t}(h^{n+1}_{K}-h^{n}_{K})+\frac{1}{\abs{K}}\sum_{\s\in{\mcal{E}(K)}}F_{\s,K}(h^{n},\uu{v}^{n})
    =0, \; \forall\, K  \in\M, \label{eq:ht_updt}\\
    \frac{1}{\delta
    t}\left(h^{n+1}_{\Ds}u^{n+1}_{\s}-h^{n}_{D_\s}u^{n}_{\s}\right)+\frac{1}{\abs{D_\s}}\sum_{\epsilon\in
    {\tilde{\mcal{E}}(D_{\s})}}F_{\epsilon,\s}(h^n
    ,\uu{v}^n)u^{n}_{\epsilon,\mathrm{up}}+(\D^{(i)}_{\mcal{E}}p^n)^{\star}_{\s}\label{eq:mom_updt}\\
    = -g(h^n\theta^n)_{\s}(\D^{(i)}_\E b)^{\star}_\s,
    \; \forall\,\s \in\E^{(i)}_\mathrm{int}, \ i=1,2,\dots,d,\nonumber \\
    \frac{1}{\dt}(h^{n+1}_K\theta^{n+1}_K-h^{n}_K\theta^n_K)
    +\frac{1}{\abs{K}}\sum_{\s\in\E(K)}G_{\s,K}(h^n,\uu{v}^n,\theta^n)=0,\;\forall\, K\in\M. \label{eq:temp_updt}
    \end{gather}
\end{subequations}
In the above, $(\D^{(i)}_\E
p^n)^{\star}_\sigma=\frac{\abs{\s}}{\abs{\Ds}}(p_L^{n,\star}-p_K^{n,\star})\uu{e}^{(i)}\cdot\uu{\nu}_{\s,K}$
for all $\s\in\E^{(i)}_\mathrm{int}, \, i=1,2,\dots,d,$ denotes the
stabilised pressure gradient, where
$p_K^{n,\star}=p_K^n-\Lambda_{K,\s}^n$ and $p_K^n=\half
g(h_K^n)^2\theta_K^n$ for all $K\in\M$. The stabilised bottom
topography term is taken as $(\D^{(i)}_\E
b)^{\star}_\sigma=(\D^{(i)}_\E b)_\sigma-(\D^{(i)}_\E S)_\sigma$ and
the stabilised velocity as $\uu{v}^n=\uu{u}^n-\delta\uu{u}^n$. Here,
the stabilisation terms $\Lambda$ and $S$ are defined on the primal
cells whereas the velocity stabilisation term $\delta\uu{u}$ is on the
dual cell. The choice of the stabilisation terms will be made aposteriori
after an energy stability analysis of the overall finite volume scheme
\eqref{eq:dis_updt} is carried out. 

\begin{remark}
  The interface values of $h^n$ and $h^n\theta^n$, defined in
  \eqref{eq:ht_interf} and \eqref{eq:tht_interf} respectively, will
  determine whether the scheme \eqref{eq:dis_updt} is referred to
  as a centred scheme or an upwind scheme. It will be shown in the
  subsequent sections that the centred scheme is energy stable,
  whereas the upwind scheme is energy consistent. Both the variants
  are well-balanced for hydrostatic steady states. The numerical
  experiments carried out in Section~\ref{sec:numerics} indicate that
  the centred scheme leads to oscillations at the discontinuities but
  the upwind scheme yields non-oscillatory results. Unless otherwise
  specified, the subsequent analysis will be carried out in
  generality as much as possible without any reference to the
  particular choice of the interface values.  
\end{remark}

\subsection{Energy stability}
\label{subsec:en_stab}
The main purpose of this section is to demonstrate the stability of
the scheme \eqref{eq:dis_updt} in terms of the discrete total energy
dissipation. The total energy of the continuous Ripa system
\eqref{eq:ripa} consists of internal, kinetic and potential
energies. Consequently, we calculate the time evolution of the
internal, the kinetic and the potential energies of the numerical
scheme in Lemma~\ref{lem:int_engy}, Lemma~\ref{lem:kin_engy} and
Lemma~\ref{lem:pot_engy}, respectively. Finally, using these
identities we obtain the global energy estimates in
Theorem~\ref{thm:tot_eng_dis_exp} and
Theorem~\ref{thm:tot_eng_dis_upw} for the centred and upwind schemes, 
respectively.  

As a first step, we average the mass update \eqref{eq:ht_updt} over
the dual cells, cf.\ also \eqref{eq:dual_avg}, to get the mass update
\begin{equation}
\label{eq:ht_balance_dual}
    \frac{1}{\delta
    t}\big(h^{n+1}_{\Ds}-h^{n}_{\Ds}\big)+\frac{1}{\abs{\Ds}}\sum_{\e\in
    {\tilde{\E}(D_{\s})}}F_{\e,\s}(h^n
    ,\uu{v}^n)= 0, \quad \forall\,\s \in\E^{(i)}_\mathrm{int}, \ i=1,2,\dots,d.
\end{equation}
Using \eqref{eq:ht_balance_dual} in the momentum balance
\eqref{eq:mom_updt} yields the velocity update  
\begin{equation}
    \label{eq:dis_vel_dual}
    \frac{u_{\s}^{n+1}-u_{\s}^n}{\dt}+\frac{1}{|\Ds|}\sum_{\e\in\tilde{\E}(\Ds)}(F_{\e,\s}^n)^{-}\frac{u_{\s^{\prime}}^{n}-u_{\s}^n}{h_{\Ds}^{n+1}}=-\frac{(\Dei p^{n})^{*}_{\s}}{h_{\Ds}^{n+1}}-g(h^n\theta^n)_\s\frac{(\Dei b)^{*}_{\s}}{h_{\Ds}^{n+1}}.
\end{equation}
Note that $a^\pm=\half(a \pm \abs{a})$, denotes the positive and
negative parts of a real number $a$. 

Next, we state three key lemmas that are crucial to obtain the energy
stability of the scheme \eqref{eq:dis_updt}. We omit the proofs as
they are analogous to the ones presented in \cite{AGK23}.  
\begin{lemma}[Internal energy identity]
  \label{lem:int_engy}
  Any solution to the system \eqref{eq:dis_updt} satisfy the
  following identity for all $K\in\M$ and
  $n\in\llbracket0,N-1\rrbracket$: 
  \begin{equation}
    \label{eq:int_eng_dis}
    \frac{\abs{K}}{\dt}\big(\half g(h_K^{n+1})^2\theta_K^{n+1}-\half g(h_K^{n})^2\theta_K^{n}\big) + \frac{g}{2}\sum_{\s\in\Ek}\abs{\s}(h^n\theta^n)_\s h_\s^n v_{\s, K}^n + \frac{g}{2}\sum_{\s\in\Ek}\abs{\s}v_{\s, K}^n (h_K^n)^2\theta_K^n=R_{K,\dt}^{n,n+1}
  \end{equation}
  where, 
  \begin{equation}
    \label{eq:int_engy_rem}
    R_{K,\dt}^{n,n+1}=\frac{g}{2}\frac{\abs{K}}{\dt}(h_K^{n+1}\theta_K^{n+1}-h_K^{n}\theta_K^{n})(h_K^{n+1}-h_K^{n})+\frac{g}{2}\sum_{\s\in\Ek}\abs{\s}((h^n\theta^n)_\s-h_K^n\theta_K^n)(h_\s^n-h_K^n) v_{\s, K}^n.
  \end{equation}
\end{lemma}
\begin{lemma}[Kinetic energy identity]
  \label{lem:kin_engy}
  Any solution to the system \eqref{eq:dis_updt} satisfy the following identity for all $\s\in\Eint$ and $n\in\llbracket0,N-1\rrbracket$:
  \begin{multline}
    \label{eq:kin_eng_dis}
    \frac{\abs{\Ds}}{2\dt}\big(h_{\Ds}^{n+1}(u_{\s}^{n+1})^2-h_{\Ds}^{n}(u_{\s}^{n})^2\big) + \sum_{\substack{\e\in\tilde{\E}(\Ds)}}F^n_{\e, \s}\frac{(u^n_{\e,\mathrm{up}})^2}{2}+\abs{\Ds}v_{\s}^n(\Dei p^n)_\s+\abs{\Ds}g(h^n\theta^n)_\s v_{\s}^n(\Dei b)_\s\\ = -\abs{\Ds}\delta u_{\s}^n((\Dei p^n)_\s+g(h^n\theta^n)_\s(\Dei b)_\s)+\abs{\s}u_\s^n(\Lambda_{L,\s}^{n}-\Lambda_{K,\s}^{n})\\
    + g\abs{\Ds}(h^n\theta^n)_\s u_{\s}^n(\Dei S)_\s+R_{\s,\dt}^{n,n+1},
  \end{multline}
  where the remainder term $\mcal{R}^{n,n+1}_{\s, \dt}$ given by
    \begin{equation}
    \label{eq:dis_kin_rem}
    \mcal{R}^{n,n+1}_{\s,\dt} =\frac{\abs{\Ds}}{2\dt}h^{n+1}_{\Ds}\abs{u^{n+1}_\s - u^n_\s}^2+\frac{1}{2}\sum_{\substack{\e\in\tilde{\E}(\Ds)\\ \e = \Ds|D_{\s^\prime}}}(F^n_{\e,\s})^{-}(u^n_{\s^\prime} - u^n_\s)^2.
    \end{equation}
\end{lemma}

\begin{lemma}[Potential energy identity]
  \label{lem:pot_engy}
  Any solution to the system \eqref{eq:dis_updt} satisfy the following
  identity for all $K\in\M$ and $n\in\llbracket 0,N-1\rrbracket$: 
  \begin{equation}
    \label{eq:pot_eng_dis}
    \frac{\abs{K}}{\dt}\big(g h_K^{n+1}\theta_K^{n+1}b_K - g h_K^{n}\theta_K^{n}b_K\big) + \sum_{\s\in\Ek}\abs{\s}g(h^n\theta^n)_\s  v_{\s, K}^n b_K = 0.
  \end{equation}
\end{lemma}
Next, we proceed with the total energy balance satisfied by the scheme
\eqref{eq:dis_updt}.
\begin{theorem}[Total energy balance of the centred scheme]
  \label{thm:tot_eng_dis_exp}
  Any solution to the scheme \eqref{eq:dis_updt} with the centred
  choice for $h_\s$ and $(h\theta)_\s$ as defined in
  Section~\ref{sec:space_discr} satisfy the following inequality for
  all $0\leqs n\leqs N-1$:   
  \begin{equation}
    \begin{aligned}
      \label{eq:tot_eng_dis}
      \sum_{K\in\M} \frac{\abs{K}}{\dt}\Big(\half g(h_K^{n+1})^2\theta_K^{n+1}-\half g(h_K^{n})^2\theta_K^{n}\Big)+\sum_{i=1}^{d}\sum_{\s\in\Eint^{(i)}}\frac{\abs{\Ds}}{\dt}\Big(\half h_{\Ds}^{n+1}(u_{\s}^{n+1})^2-\half h_{\Ds}^{n}(u_{\s}^{n})^2\Big)\\
      +\sum_{K\in\M}\frac{\abs{K}}{\dt}\big(g h_K^{n+1}\theta_K^{n+1}b_K - g h_K^{n}\theta_K^{n}b_K\big)\leqs 0,
    \end{aligned}
  \end{equation}
  under the following conditions.
  \begin{enumerate}[label=(\roman*)]
  \item A CFL type condition on the time-step:
    \begin{equation}
        \label{eq:cfl}
        \dt\leqs\min\Big\{\frac{h_{\Ds}^{n+1}\abs{\Ds}}{4\sum_{\e\in\tilde{\E}(\Ds)}(-(F^n_{\e,\s})^{-})},\Big(\frac{\alpha-\frac{g}{2}}{4\alpha^2 a_K^n}\Big)^{\half},\Big(\frac{\beta-\frac{1}{2}}{\beta^2 b_K^n}\Big)^{\half}, \Big(\frac{1}{\eta_{\s}^2 c_{\s}^n}\big(\eta_\s-\frac{2}{h_{\Ds}^{n+1}}\big)\Big)^{\half}\Big\},
    \end{equation}
    with 
    \begin{equation}
    \label{eq:eqn_ansigma}
        a_K^n = \frac{1}{\abs{K}}\sum_{\s\in\Ek}\frac{\abs{\s}^2}{\abs{\Ds}}\frac{(h_{\s}^n)^2}{h^{n+1}_{\Ds}}, \, b_K^n = \frac{1}{\abs{K}}\sum_{\s\in\Ek}\frac{\abs{\s}^2}{\abs{\Ds}}\frac{g(h^n\theta^n)_{\s}^2}{h^{n+1}_{\Ds}}, \, c_{\s}^n = \frac{2(1+c_\theta)}{\Delta_\s}\frac{\abs{\s}}{\abs{\Ds}}(h^n_{\s})^2,
    \end{equation}
    and $\alpha>\frac{g}{2}$, $\beta>\frac{1}{2}$ and
    $c_\theta=\max_{K\in\M}\theta_K^n$.  
    \item The following choices for the stabilisation terms
      $\delta\uu{u}^{n}$, $\Lambda^{n}$ and $S^{n}$: 
    \begin{subequations}
    \label{eq:stab_term}
        \begin{align}
        \label{eq:dis_delu}
        \delta u_{\s}^n &=\eta_\s\dt\{(\Dei p^n)_{\s}+g(h^n\theta^n)_\s(\Dei b)_{\s}\},\;\forall\s\in\E^{(i)}_\intr,\; 1\leqs i \leqs d, \\
        \label{eq:dis_S}
        S_K^n &=\beta\frac{\dt}{|K|}\sum_{\s\in\Ek}|\s| (h^n\theta^n)_\s u_{\s,K}^n,\;\forall K\in\M,\\
        \label{eq:dis_Lambda}
        \Lambda_{K,\s}^n &=\alpha h_\s^n\frac{\dt}{|K|}\sum_{\s\in\Ek}|\s| h_\s^n u_{\s,K}^n,\;\forall K\in\M,
    \end{align} 
    \end{subequations}  
\end{enumerate}
with $\eta_\s>\frac{2}{h_{\Ds}^{n+1}}$ for each $\s\in\E^{(i)}_\intr,
\ 1\leqs i \leqs d$ and a geometric constant $\frac{1}{\Delta_\s}$
defined as 
$\frac{1}{\Delta_\s}=\half(\frac{\abs{\partial
    K}}{\abs{K}}+\frac{\abs{\partial L}}{\abs{L}})$. 
\end{theorem}
\begin{proof}
  We begin by summing the equations \eqref{eq:int_eng_dis},
  \eqref{eq:kin_eng_dis} and \eqref{eq:pot_eng_dis} over respective
  control volumes. The second term in the remainder \eqref{eq:int_engy_rem} can be dropped as it is locally conservative due to the interface values of height and temperature. Exploiting the div-grad duality \eqref{eq:dis_dual} and the
  local conservation of convective fluxes, we subsequently obtain the
  following identity:       
  \begin{equation}
    \label{eq:pf_tot-eng}
    \begin{aligned}
      \sum_{K\in\M} \frac{\abs{K}}{\dt}\Big(\half g(h_K^{n+1})^2\theta_K^{n+1}-\half g(h_K^{n})^2\theta_K^{n}\Big)+\sum_{i=1}^{d}\sum_{\s\in\Eint^{(i)}}\frac{\abs{\Ds}}{\dt}\Big(\half h_{\Ds}^{n+1}(u_{\s}^{n+1})^2-\half h_{\Ds}^{n}(u_{\s}^{n})^2\Big)\\
      +\sum_{K\in\M}\frac{\abs{K}}{\dt}\big(g h_K^{n+1}\theta_K^{n+1}b_K - g h_K^{n}\theta_K^{n}b_K\big)=\sum_{K\in\M}\frac{g}{2}\frac{\abs{K}}{\dt}(h_K^{n+1}\theta_K^{n+1}-h_K^{n}\theta_K^{n})(h_K^{n+1}-h_K^{n})\\
      -\sum_{i=1}^{d}\sum_{\s\in\Eint^{(i)}}\abs{\Ds}\delta u_{\s}^n\big((\Dei
      p^n)_\s+g(h^n\theta^n)_\s(\Dei
      b)_\s\big)+\sum_{i=1}^{d}\sum_{\s\in\Eint^{(i)}}\abs{\s}u_\s^n(\Lambda_{L,\s}^{n}-\Lambda_{K,\s}^{n})\\ + \sum_{i=1}^{d}\sum_{\s\in\Eint^{(i)}}g\abs{\Ds}(h^n\theta^n)_\s u_{\s}^n(\Dei S^n)_\s+\sum_{i=1}^{d}\sum_{\s\in\Eint^{(i)}}\mcal{R}^{n,n+1}_{\s,\dt}.
    \end{aligned}
  \end{equation}
  In order to prove the required inequality \eqref{eq:tot_eng_dis}, it
  is enough to show that right side of the above relation
  \eqref{eq:pf_tot-eng} is non-positive. To this end, we start by
  evaluating the net remainder term on the right hand side of
  \eqref{eq:pf_tot-eng}.  
  \begin{align}
    \label{eq:dis_rhs_rem}
    &\sum_{\s\in\Eint^{(i)}}\mcal{R}^{n,n+1}_{\s,\dt}\leqs \dt\sum_{\s\in\Eint^{(i)}}\abs{\Ds}\frac{2}{h^{n+1}_{\Ds}}\big((\Dei p^n)_\s+g(h^n\theta^n)_\s(\Dei b)_\s\big)^2\notag\\
    &+\dt\sum_{\s\in\Eint^{(i)}}\frac{\absq{\s}}{\abs{\Ds}}\frac{2}{h^{n+1}_{\Ds}}((\Lambda_{L,\s}^{n}-\Lambda_{K,\s}^{n})^2+\dt\sum_{\s\in\Eint^{(i)}}\abs{\Ds}\frac{2}{h^{n+1}_{\Ds}}(g(h^n\theta^n)_\s)^2(\Dei S^n)_{\s}^2\\
    &+\sum_{\s\in\Eint^{(i)}}\frac{2\dt}{\abs{\Ds}h^{n+1}_{\Ds}}\Big(\sum_{\e\in\tilde{\E}(\Ds)}(F^n_{\e,\s})^{-}(u^n_{\s^{\prime}} - u^n_\s)\Big)^2 +\frac{1}{2}\sum_{\s\in\Eint^{(i)}}\sum_{\substack{\e\in\tilde{\E}(\Ds)\\ \e = \Ds|D_{\s^\prime}}}(F^n_{\e,\s})^{-}(u^n_{\s^\prime} - u^n_\s)^2.\notag
  \end{align}
  Under a $\CFL$ condition of the form
  \begin{equation}
    \label{eq:cfl_pf}
    \dt\leqs\frac{h_{\Ds}^{n+1}\abs{\Ds}}{4\sum_{\e\in\tilde{\E}(\Ds)}(-(F^n_{\e,\s})^{-})},
  \end{equation}
  we can recast \eqref{eq:dis_rhs_rem} as
  \begin{multline}
    \sum_{\s\in\Eint^{(i)}}\mcal{R}^{n,n+1}_{\s,\dt}\leqs \dt\sum_{\s\in\Eint^{(i)}}\abs{\Ds}\frac{2}{h^{n+1}_{\Ds}}\big((\Dei p^n)_\s+g(h^n\theta^n)_\s(\Dei b)_\s\big)^2\\
    +\dt\sum_{\s\in\Eint^{(i)}}\frac{\absq{\s}}{\abs{\Ds}}\frac{2}{h^{n+1}_{\Ds}}((\Lambda_{L,\s}^{n}-\Lambda_{K,\s}^{n})^2+\dt\sum_{\s\in\Eint^{(i)}}\abs{\Ds}\frac{2}{h^{n+1}_{\Ds}}(g(h^n\theta^n)_\s)^2(\Dei S^n)_{\s}^2\\
    +\sum_{\s\in\Eint^{(i)}}\frac{2\dt}{\abs{\Ds}h^{n+1}_{\Ds}}\Big(\sum_{\e\in\tilde{\E}(\Ds)}(F^n_{\e,\s})^{-}(u^n_{\s^{\prime}} - u^n_\s)\Big)^2 +\frac{1}{2}\sum_{\s\in\Eint^{(i)}}\sum_{\substack{\e\in\tilde{\E}(\Ds)\\ \e = \Ds|D_{\s^\prime}}}(F^n_{\e,\s})^{-}(u^n_{\s^\prime} - u^n_\s)^2.
  \end{multline}
  Next, we start with the summand in the first term on the right hand
  side of \eqref{eq:pf_tot-eng} and write 
  \begin{multline}
    \frac{\abs{K}}{\dt}(h_K^{n+1}\theta_K^{n+1}-h_K^{n}\theta_K^{n})(h_K^{n+1}-h_K^{n})\leqs \frac{\dt}{\abs{K}}\Big(\sum_{\s\in\Ek}\abs{\s}h^n_\s u^n_{\s,K}\Big)^2+\frac{\dt}{\abs{K}}\Big(\sum_{\s\in\Ek}\abs{\s}(h^n\theta^n)_\s u^n_{\s,K}\Big)^2\\
    +(1+c_\theta)\dt\frac{\abs{\partial K}}{\abs{K}}\sum_{\s\in\Ek}\abs{\s}(h^n_\s)^2 (\delta u^n_{\s,K})^2.
  \end{multline}
  Injecting the above two estimates into \eqref{eq:pf_tot-eng} yields
  \begin{equation}
    \label{eq:pf_tot-eng_m}
    \begin{aligned}
      \sum_{K\in\M} \frac{\abs{K}}{\dt}\big(\half g(h_K^{n+1})^2\theta_K^{n+1}-\half g(h_K^{n})^2\theta_K^{n}\big)+\sum_{i=1}^{d}\sum_{\s\in\Eint^{(i)}}\frac{\abs{\Ds}}{\dt}\big(\half h_{\Ds}^{n+1}(u_{\s}^{n+1})^2-\half h_{\Ds}^{n}(u_{\s}^{n})^2\big)\\
      +\sum_{K\in\M}\frac{\abs{K}}{\dt}\big(gh_K^{n+1}\theta_K^{n+1}b_K - g h_K^{n}\theta_K^{n}b_K\big)\leqs\mcal{A}+\mcal{R}+\mcal{Q},
    \end{aligned}
  \end{equation}
  where
  \begin{subequations}
    \begin{multline}
      \mcal{A}=\sum_{i=1}^{d}\sum_{\s\in\Eint^{(i)}}\abs{\s} u_\s^n(\Lambda^{n}_{L,\s}-\Lambda^{n}_{K,\s})+\dt\sum_{i=1}^{d}\sum_{\s\in\Eint^{(i)}}\frac{\absq{\s}}{\abs{\Ds}}\frac{2}{h^{n+1}_{\Ds}}(\Lambda^{n}_{L,\s}-\Lambda^{n}_{K,\s})^2\\
      +\frac{g\dt}{2}\sum_{K\in\M}\abs{K}\Big(\frac{1}{\abs{K}}\sum_{\s\in\Ek}\abs{\s}h^n_{\s}u^n_{\s,K}\Big)^2,
    \end{multline}
    \begin{multline}
      \mcal{R}=-\sum_{i=1}^{d}\sum_{\s\in\Eint^{(i)}}\abs{\Ds}\delta u_{\s}^n((\Dei p^n)_\s+g(h^n\theta^n)_\s(\Dei b)_\s)\\
      +\dt\sum_{i=1}^{d}\sum_{\s\in\Eint^{(i)}}\abs{\Ds}\frac{2}{h^{n+1}_{\Ds}}\big((\Dei
      p^n)_\s+g(h^n\theta^n)_\s(\Dei b)_\s\big)^2\\
      +\dt(1+c_\theta)\sum_{K\in\M}\frac{\abs{\partial K}}{\abs{K}}\sum_{\s\in\Ek}\abs{\s}(h^{n}_{\s})^2(\delta u^n_{\s, K})^2,
    \end{multline}
    \begin{multline}
      \mcal{Q}=\sum_{i=1}^{d}\sum_{\s\in\Eint^{(i)}}\abs{\Ds} g(h^n\theta^n)_{\s}u_\s^n(\Dei S^{n})_\s+\dt\sum_{i=1}^{d}\sum_{\s\in\Eint^{(i)}}\abs{\Ds}\frac{2}{h^{n+1}_{\Ds}}(g(h^n\theta^n)_{\s})^2((\Dei S^n)_\s)^2\\
      +\frac{g\dt}{2}\sum_{K\in\M}\abs{K}\Big(\frac{1}{\abs{K}}\sum_{\s\in\Ek}\abs{\s}(h^n\theta^n)_{\s}u^n_{\s,K}\Big)^2.
    \end{multline}
  \end{subequations}
  Using the choice of the stabilisation terms as defined in
  \eqref{eq:stab_term} and applying the div-grad duality
  \eqref{eq:dis_dual}, we can further simplify the above terms as
  \begin{subequations}
    \begin{alignat}{3}
      \mcal{A}&=\dt\sum_{K\in\M}\abs{K}\big(4a_K^n(\dt)^2\alpha^2-\alpha+\frac{g}{2}\big)\Big(\frac{1}{\abs{K}}\sum_{\s\in\Ek}\abs{\s}h^n_{\s}u^n_{\s,K}\Big)^2, \label{eq:A}\\
      \mcal{R}&=\dt\sum_{i=1}^{d}\sum_{\s\in\Eint^{(i)}}\abs{\Ds}\Big(c_\s^n(\dt)^2\eta_\s^2-\eta_\s+\frac{2}{h^{n+1}_{\Ds}}\Big)\big((\Dei
      p^n)_\s+g(h^n\theta^n)_\s(\Dei b)_\s\big)^2, \label{eq:R}\\
      \mcal{Q}&=g\dt\sum_{K\in\M}\abs{K}\big(4b_K^n(\dt)^2\beta^2-\beta+\frac{1}{2}\big)\Big(\frac{1}{\abs{K}}\sum_{\s\in\Ek}\abs{\s}(h^n\theta^n)_{\s}u^n_{\s,K}\Big)^2. \label{eq:Q}
    \end{alignat}
  \end{subequations}
  Finally, using the timestep condition \eqref{eq:cfl} and the
  specific choices of the parameters $\alpha,\beta$ and $\eta_\s$ as
  outlined in the theorem, we can easily see that the quadratic polynomials $\mcal{A}$, $\mcal{R}$ and $\mcal{Q}$ are non-positive. 
\end{proof}

\begin{theorem}[Total energy balance of the upwind scheme]
  \label{thm:tot_eng_dis_upw}
  Under the timestep condition and the choice of the stabilization
  terms as in Theorem~\ref{thm:tot_eng_dis_exp}, any solution to the
  scheme \eqref{eq:dis_updt} with the upwind choice for $h_\s$ and
  $(h\theta)_\s$ as defined in Section~\ref{sec:space_discr} satisfies
  the following inequality for all $0\leqs n\leqs N-1$: 
  \begin{equation}
    \begin{aligned}
      \label{eq:tot_eng_dis_upw}
      \sum_{K\in\M} \frac{\abs{K}}{\dt}\big(\half
      g(h_K^{n+1})^2\theta_K^{n+1}-\half
      g(h_K^{n})^2\theta_K^{n}\big)+\sum_{i=1}^{d}\sum_{\s\in\Eint^{(i)}}\frac{\abs{\Ds}}{\dt}\big(\half
      h_{\Ds}^{n+1}(u_{\s}^{n+1})^2-\half
      h_{\Ds}^{n}(u_{\s}^{n})^2\big)\\ 
      +\sum_{K\in\M}\frac{\abs{K}}{\dt}\big(g
      h_K^{n+1}\theta_K^{n+1}b_K - g
      h_K^{n}\theta_K^{n}b_K\big)+\sum_{K\in\M}\frac{g}{2}\sum_{\s\in\Ek}\abs{\s}((h^n\theta^n)_\s-h_K^n\theta_K^n)(h_K^n-h_\s^n)
      v_{\s, K}^n\leqs 0. 
    \end{aligned}
  \end{equation}
\end{theorem}
In contrast to the centred scheme, we are unable to prove total energy
dissipation for the upwind scheme as the sign of the last term in the
inequality \eqref{eq:tot_eng_dis_upw} cannot be determined. Instead,
we show in Section~\ref{sec:cons_LW} below that that the upwind scheme
is entropy consistent as this reminder term vanishes when the mesh
parameters tend to zero.  

\subsection{Structure preserving property}
\label{subsec:positivity}
Our goal is now to establish the structure preserving property of the
finite volume scheme, i.e.\ to show that it preserves the positivity of the
water height and the potential temperature. We achieve this by
imposing a $\CFL$-type condition on the timestep.
\begin{proposition}
  Let $0\leqs n \leqs N-1$ and $(h^n,\uu{u}^n,\theta^n)\in
  L_{\mcal{M}}(\Omega)\times\uu{H}_{\mcal{E},0}(\Omega)\times
  L_\M(\Omega)$ be given so that $h_K^n>0$ and $\theta_K^n>0$ $\forall
  K\in\M$. Then $h_K^{n+1}>0$ and $\theta_K^{n+1}>0$ $\forall K\in\M$
  under a timestep $\dt>0$ such that for each $\s\in\E^{(i)},\,
  i\in\{1,\dots,d\},\, \s = K|L$, the following holds:  
\begin{equation}
  \label{eq:suff_tstep_positivity}
  \dt\max\Bigg\{\frac{\abs{\D K}}{\abs{K}},\frac{\abs{\D
      L}}{\abs{L}}\Bigg\}\Bigg(\abs{u^n_\s} +
  \sqrt{\tilde{\eta}_\s\left|p^{n}_L - p^{n}_K+g(h^n\theta^{n})_\s(b_L -
      b_K)\right|}\Bigg)\leq \frac{1}{5}\mu^{n}_{K,L}, 
\end{equation}
where
\begin{equation}
  \label{eq:mu_KL}
  \tilde{\eta}_\s\max\bigg\{\frac{\abs{\D K}}{\abs{K}},\frac{\abs{\D L}}{\abs{L}}\bigg\}=\eta_\s\frac{\abs{\s}}{\abs{\Ds}} \, \text{ and } \, \mu^{n}_{K,L}=\frac{\min\{h^n_K,h^n_L\}}{h^{n}_\s}\frac{\min\{\theta^n_K,\theta^n_L\}}{\max\{\theta^n_K,\theta^n_L\}}.
\end{equation}
\end{proposition}
\begin{proof}
  Note that the inequality \eqref{eq:suff_tstep_positivity} implies
    \begin{equation}
    \label{eq:timestep_height}
    \dt\max\Bigg\{\frac{\abs{\D K}}{\abs{K}},\frac{\abs{\D
        L}}{\abs{L}}\Bigg\}\Bigg(\abs{u^n_\s} +
    \sqrt{\tilde{\eta}_\s\left|p^{n}_L -
        p^{n}_K+g(h^n\theta^{n})_\s(b_L - b_K)\right|}\Bigg)\leq
    \frac{1}{5}\frac{\min\{h^n_K,h^n_L\}}{h^{n}_\s}, 
    \end{equation}
    and
\begin{equation}
\label{eq:timestep_posit_temp}
\dt\max\Bigg\{\frac{\abs{\D K}}{\abs{K}},\frac{\abs{\D
    L}}{\abs{L}}\Bigg\}\Bigg(\abs{u^n_\s} +
\sqrt{\tilde{\eta}_\s\left|p^{n}_L - p^{n}_K+g(h^n\theta^{n})_\s(b_L -
    b_K)\right|}\Bigg)\leq
\frac{1}{5}\frac{\min\{h^n_K\theta^n_K,h^n_L\theta^n_L\}}{(h^{n}\theta^n)_\s}. 
\end{equation}
Since the right hand side of \eqref{eq:timestep_height} is less than
$1$, it leads to
\begin{equation}
    \dt\max\Bigg\{\frac{\abs{\D K}}{\abs{K}},\frac{\abs{\D
        L}}{\abs{L}}\Bigg\}\bigg(\abs{u^n_\s}+\dt\eta_\s\big|(\Dei
    p^{n})_\s+g(h^n\theta^{n})_\s(\Dei b)_\s\big|\bigg)\leq
    \frac{1}{5}\frac{\min\{h^n_K,h^n_L\}}{h^{n}_\s}. 
\end{equation}

Now we use the mass balance \eqref{eq:ht_updt} to get
\begin{equation}
    h_{K}^{n+1}-\frac{4}{5}h_{K}^{n}\geqslant
    h_{K}^{n+1}-h_{K}^{n}+\dfrac{\dt}{|K|}\sum_{\s\in\Ek}\abs{F^n_{\s,K}}\geqslant
    0. 
\end{equation}
A similar procedure can be followed to prove the positivity of
$h_{K}^{n+1}\theta_{K}^{n+1}$ from the temperature balance
\eqref{eq:timestep_posit_temp}.  
\end{proof}

\begin{remark}
  The timestep restriction \eqref{eq:suff_tstep_positivity} provides
  a sufficient condition which in turn implies the CFL condition
  \eqref{eq:cfl_pf} needed for stability.  
\end{remark}
\begin{remark}
  Under the sufficient timestep condition
  \eqref{eq:suff_tstep_positivity}, the positivity of the water depth
  and the temperature is still preserved also when we use the upwind
  flux instead of the centred flux.   
\end{remark}

\subsection{Well-balancing property}
\label{subsec:wb}
In the following theorem we prove that both the centred and upwind
schemes can exactly preserve the hydrostatic steady states
\eqref{eq:steady_states} at a discrete level.    
\begin{theorem}
  Assume that the initial data $(h^0,\uu{u}^0,\theta^0)\in
  L_{\M}(\Omega)\times \Hez\times L_{\M}(\Omega)$ is given by one of the
  following discrete hydrostatic steady states, i.e.\ for all $K\in\M$ and
  $\s\in\E^{(i)}, \, 1\leqs i\leqs d$: 
  \begin{equation}
    \begin{aligned}
      \begin{cases}
        u_{\s}^0 = 0, \\
        \theta_K^0 = \Theta, \\
        h_K^0 + b_K = H,
      \end{cases}
      \quad \text{or} \quad
      \begin{cases}
        u_{\s}^0 = 0, \\
        b_K = B, \\
        (h_K^0)^2 \theta_K^0 = P,
      \end{cases}
      \quad \text{or} \quad
      \begin{cases}
        u_{\s}^0 = 0, \\
        h_K^0 = H, \\
        b_K + \frac{h_K^0}{2} \ln(\theta_K^0) = P,
      \end{cases}
    \end{aligned}
  \end{equation}
  where $\Theta,H,B$ and $P$ are constants. Then the numerical solution
  $(h^{n+1},\uu{u}^{n+1},\theta^{n+1})\in L_{\M}(\Omega)\times
  \Hez\times L_{\M}(\Omega)$ given by the scheme \eqref{eq:dis_updt}
  stays constant in time, i.e.\
  $(h^{n+1},\uu{u}^{n+1},\theta^{n+1})=(h^{n},\uu{u}^{n},\theta^{n})$,
  for $0\leqs n\leqs N-1$. 
\end{theorem}
\begin{proof}
  Thanks to the interface choice $(h^0\theta^0)_\s$ as defined in
  \eqref{eq:tht_interf}, it is immediate that $\delta
  u^0_\s=0$. Since the velocities are zero for all the steady states,
  the convective fluxes in the updates \eqref{eq:dis_updt}
  vanish. Subsequently, $h^1_{\Ds} u^1_{\s}=0$ for all $\s\in\Eint$,
  from the momentum update \eqref{eq:mom_updt} and
  $(h^1_K,\theta^1_K)=(h^0_K,\theta^0_K)$ for all $K\in\M$, from the
  mass and temperature updates \eqref{eq:ht_updt} and
  \eqref{eq:temp_updt}. Thus, we obtain that
  $(h^1,\uu{u}^1,\theta^1)=(h^0,\uu{u}^0,\theta^0)$. The proof of the
  theorem now follows from an induction argument. 
\end{proof}

\section{Weak Consistency of the Scheme}
\label{sec:cons_LW}
This section aims to prove the weak consistency of the proposed scheme
\eqref{eq:dis_updt} in the sense of Lax-Wendroff. We show that if a sequence of approximate solutions generated by successive mesh
refinements remains bounded under suitable norms and converges
strongly, then the limit must be an entropy weak solution of the Ripa
system \eqref{eq:ripa}. 
\begin{definition}
  A triple $(h, \uu{u},\theta) \in L^\infty([0, T)\times\Omega))^{d+2}$
  is a weak solution to the Ripa system \eqref{eq:ripa} if it
  satisfies the following identities for test functions  
  $\varphi \in C^\infty_c([0, T)\times\Omega)$ and $\uu{\varphi}\in
  C^\infty_c([0, T)\times\Omega)^d$. 
  \begin{subequations}
    \label{eq:weak_soln_ripa}
    \begin{equation}
      \int_0^T \int_{\Omega} \left[ h \partial_t \varphi + h \uu{u}
        \cdot \nabla \varphi \right] \, \dd\uu{x} \, \dd t  
      + \int_{\Omega} h_0(\uu{x}) \varphi(0, \uu{x}) \, \dd\uu{x} = 0.
    \end{equation}
    \begin{multline}
      \int_0^T \int_{\Omega} \Big[ h \uu{u} \cdot \partial_t
      \uu{\varphi} + (h \uu{u} \otimes \uu{u}) : \nabla \uu{\varphi} +
      \frac{1}{2} g h^2\theta \div \uu{\varphi} + g h\theta \nabla b
      \cdot \uu{\varphi} \Big] \, \dd\uu{x} \, \dd t\\
      + \int_{\Omega} h_0(\uu{x}) \uu{u}_0(\uu{x}) \cdot \uu{\varphi}(0,
      \uu{x}) \, \dd\uu{x} = 0. 
    \end{multline}
    \begin{equation}
      \int_0^T \int_{\Omega} \left[ h\theta \partial_t \varphi + h\theta
        \uu{u} \cdot \nabla \varphi \right] \, \dd\uu{x} \, \dd t  
      + \int_{\Omega} h_0(\uu{x})\theta_0(\uu{x})\varphi(0, \uu{x})
      \,\dd\uu{x} = 0. 
    \end{equation}
  \end{subequations}
  
  A weak solution of the system \eqref{eq:ripa} is called an
  entropy weak solution if, for any non-negative test function $\varphi
  \in C^\infty_c([0, T)\times\Omega;\mbb{R}^+)$, the following
  inequality holds: 
  \begin{equation}
    \label{eq:entropy_ineq}
    \int_0^T \int_{\Omega}\Big[E \partial_t \varphi + \big(E +
    \frac{1}{2} g h^2\theta \big) \uu{u} \cdot 
    \nabla \varphi\Big] \, \dd\uu{x} \, \dd t + \int_{\Omega} E_0(\uu{x})
    \varphi(0, x) \, \dd\uu{x} \geqslant 0, 
  \end{equation}
  where $E = \frac{1}{2} h |\uu{u}|^2 + \frac{1}{2} g h^2\theta + g
  h\theta b$ and $E_0 = \frac{1}{2} h_0 |\uu{u}_0|^2 + \frac{1}{2} g
  h_0^2\theta_0 + g h_0\theta_0 b$.
\end{definition}
Next, we proceed to prove the main result of this section, namely the
weak consistency the scheme \eqref{eq:dis_updt} and its consistency
with the energy inequality \eqref{eq:entropy_ineq}. The proof uses
analogous calculations done in, e.g.\
\cite{AGK23,AGK24,AK24,GHL22,HLN+23} and hence we omit most of the
technical details. 
\begin{theorem}
\label{thm:weak_cons}
Let $\Omega$ be an open, bounded set of $\mbb{R}^d$. Assume that
$\big(\M^{(m)},\E^{(m)})_{m\in\mbb{N}}$ is a sequence of MAC grids and 
$\delta t^{(m)}$ is a sequence of timesteps such that both
$\lim_{m\rightarrow \infty}\delta t^{(m)}$ and $\lim_{m\rightarrow
  \infty}\delta_{\M^{(m)}}$ are $0$, where $\delta_{\M^{(m)}} =
\max_{K\in\M^{(m)}}\diam(K)$. Let
$\big(h^{(m)},\uu{u}^{(m)},\theta^{(m)}\big)_{m\in\mbb{N}}$ be the
corresponding sequence of discrete solutions with respect to an initial datum $(h_0,\uu{u}_0,\theta_{0})\in L^\infty(\Omega)^{d+2}$. We
assume that $(h^{(m)},\uu{u}^{(m)}, \theta^{(m)})_{m\in\mbb{N}}$
satisfies the following. 
\begin{enumerate}[label=(\roman*)]
\item $\big(h^{(m)},\uu{u}^{(m)},\theta^{(m)}\big)_{m\in\mbb{N}}$ is
  uniformly bounded in $L^\infty(Q)^{1+d}$, i.e.\ 
\begin{subequations}
\label{eq:solu_abs_bound}
\begin{align}
\label{eq:ht_abs_bound}
\ubar{C}<(h^{(m)})^n_K \leqslant \bar{C}, \ \forall
  K\in\mcal{M}^{(m)}, \ 0\leqslant n\leqslant N^{(m)}, \ \forall
  m\in\mbb{N}, 
\end{align}
\begin{align}
\label{eq:temp_abs_bound}
  \ubar{C}<(\theta^{(m)})^n_K &\leqslant \bar{C}, \ \forall
                                K\in\mcal{M}^{(m)}, \ 0\leqslant
                                n\leqslant N^{(m)}, \ \forall
                                m\in\mbb{N},
\end{align}
\begin{align}
  \label{eq:u_abs_bound}
  |(u^{(m)})^n_\sigma| &\leqslant C, \ \forall
                         \sigma\in\mcal{E}^{(m)}, \ 0\leqslant
                         n\leqslant N^{(m)}, \ \forall m\in\mbb{N}.  
\end{align}
\end{subequations}
where $\ubar{C}, \bar{C}, C>0$ are constants independent of the
discretisations.  
\item $\big(h^{(m)},\uu{u}^{(m)}, \theta^{(m)}\big)_{m\in\mbb{N}}$
  converges to $(h,\uu{u}, \theta)\in L^\infty(0,
  T;L^\infty(\Omega)^{1+d+1})$ in $L^r(Q_T)^{1+d+1}$ for $1\leqslant
  r<\infty$.  
\end{enumerate}
Furthermore, assume that the sequence of grids
$\big(\M^{(m)},\E^{(m)}\big)_{m\in\mbb{N}}$ and the timesteps $\delta 
t^{(m)}$ satisfies the regularity conditions: 
\begin{equation}
\label{eq:CFL_restric}
\frac{\delta t^{(m)}}{\min_{K\in\mcal{M}^{(m)}}\abs{K}}\leqslant\mu,\
\max_{K \in \mcal{M}^{(m)}}
\frac{\diam(K)^2}{\abs{K}}\leqslant\mu,\;\forall m\in\mbb{N}, 
\end{equation}
where $\mu>0$ is independent of the discretisations. Then
$(h,\uu{u},\theta)$ satisfies the weak formulation
\eqref{eq:weak_soln_ripa}. Moreover, $(h,\uu{u},\theta)$ satisfies the
entropy inequality \eqref{eq:entropy_ineq}. 
\end{theorem}
\begin{proof}
  For the first part of the proof, i.e.\ to show that the limit
  $(h,\uu{u},\theta)$ satisfies the weak formulation
  \eqref{eq:weak_soln_ripa}, we proceed as in \cite[Theorem
  5.2]{AGK23} and hence omit the details to save the space. 
  
  Next, we turn to the second part of the theorem, i.e.\ proving the
  energy consistency. Let $\varphi \in C^\infty_c([0, T)\times\Omega,
  \mathbb{R}^+)$ denote a non-negative test function and denote by
  $\varphi^n_K$ (resp.\ $\varphi^n_\s$), the mean value of $\varphi$
  on $(t^n, t^{n+1})\times K$ (resp.\ $(t^n, t^{n+1})\times\Ds$), for
  any $K \in \M^{(m)}$ (resp.\ $\s \in \E^{(m)}$) and $n \in \{0,
  \dots, N^{(m)} - 1\}$. We multiply the internal energy identity
  \eqref{eq:int_eng_dis} and the potential energy identity
  \eqref{eq:pot_eng_dis} by $\dt\varphi_K^n$ and sum over all 
  $K\in\M^{(m)}$. Similarly, we multiply the kinetic energy identity
  \eqref{eq:kin_eng_dis} by $\dt\varphi_\s^n$ and sum over all
  $\s\in\E^{(m)}$. After summing the resulting expressions and
  adopting similar techniques as in the proof of \cite[Theorem
  5.1]{HLN+23}, it can be shown that we can recover the energy inequality \eqref{eq:entropy_ineq}. It only remains to show
  that the contribution from the remainder term, namely the last term
  in \eqref{eq:tot_eng_dis_upw}, goes to zero as $m\to\infty$. Note that the
  remainder term is given by 
  \begin{equation}
    R^{(m)}=\frac{g}{2}\sum_{n=0}^{N^{(m)}-1}\dt^{(m)}\sum_{K\in\M^{(m)}}\sum_{\s\in\E^{(m)}(K)}\abs{\s}((h^n\theta^n)_\s-h_K^n 
    \theta_K^n)(h_\s^n-h_K^n)v^n_{\s,K}\varphi^n_K. 
  \end{equation}
  Invoking the $L^\infty$ boundedness assumptions
  \eqref{eq:solu_abs_bound} of the solutions, the right hand side can be
  estimated as 
  \begin{equation}
    R^{(m)}\leqs\frac{g}{2}\norm{\varphi}_{\infty}\tilde{C}\sum_{n=0}^{N^{(m)}-1}\dt^{(m)}\sum_{K\in\M^{(m)}}\sum_{\s\in\E^{(m)}(K)}\abs{\s}\abs{h_\s^n-h_K^n},
  \end{equation}
  Now, for $h^n_\s\in\llbracket h_K^n,h_L^n\rrbracket$, one can
  further obtain $\abs{h_\s^n-h_K^n}\leqs\abs{h_L^n-h_K^n}$. Finally,
  using analogous techniques as in \cite[Lemma A.6]{HLN+23} regarding
  the limit of discrete space translates, we observe that $R^{(m)}$
  tends to $0$ as $m\to\infty$ based on the $L^1$ convergence of
  $h^{(m)}$ and the regularity assumptions \eqref{eq:CFL_restric} on
  the sequence of meshes. This concludes the proof. 
\end{proof}

\section{Numerical Results}
\label{sec:numerics}
 In this section, we analyse the performance of the proposed scheme
 through several $1d$ and $2d$ numerical case studies. In all our case
 studies, we take the gravitational constant as $g=1$. The timestep
 condition is given by \eqref{eq:cfl} and
 \eqref{eq:suff_tstep_positivity}. Unless otherwise specified, we use
 the proposed explicit upwind scheme for the numerical experiments
 performed below. 

\subsection{Steady state solution}
\label{subsec:ss_test}
The goal of this test problem is to assess the well-balancing property
of the proposed scheme in preserving the hydrostatic equilibria. We take
initial data corresponding to the three hydrostatic steady states as
given in \cite{FN20}. The initial velocity is $u(0,x)=0$, the bottom
topography $b$, the initial water height $h_0$ and the temperature
$\theta_0$ are given as follows.
\subsubsection{Lake at rest steady state}
\begin{equation}
  b(x) = 0.1 + G_{x_0,\sigma}(x), \quad h_0(x) = 8.0 - b(x), \quad \theta_0(x) = 1.
\end{equation}
\subsubsection{Isobaric steady state}
\begin{equation}
  b(x) = 1, \quad h_0(x) = 1.0 + 0.2G_{x_0,\sigma}(x), \quad
  \theta_0(x) = \frac{1}{h_0(x)^2}. 
\end{equation}
\subsubsection{Constant height steady state}
\begin{equation}
  b(x) = 10+x(1-x), \quad h_0(x) = 1, \quad \theta_0(x) = 0.1e^{-2b(x)}.
\end{equation}
Here, $G_{x_0,\s}(x) = \frac{1}{\sqrt{2\pi \sigma}}
\exp\big(-\frac{(x - x_0)^2}{\sigma}\big)$, with $x_0=0.5$ and  $\s
= 0.06$. 
The computational domain $[0,3]$ is discretised using $200$ grid
points. The simulations are run up to a final time $T = 20$. In
Table~\ref{tab:l1_ss}, we present $L^1$ errors in the water height,
the velocity and the temperature with respect to the steady state. In
order to compare the performance of the well-balanced scheme, we also
present the errors obtained using a non well-balanced Rusanov
scheme. From the data, we note that present scheme maintains all the
three stead states very well, and the errors obtained are much smaller
than that of the Rusanov scheme.   
\begin{table}[h!]
\centering
\renewcommand{\arraystretch}{1.2} 
\begin{tabular}{|c|c|c|c|c|c|c|}
\hline
 Steady state   & \multicolumn{2}{c|}{$L^1$ error in $h$} &
                                                            \multicolumn{2}{c|}{$L^1$ error in $u$} & \multicolumn{2}{c|}{$L^1$ error in $\theta$} \\ \hline 
    & {Rusanov} & {centred} & {Rusanov} & {centred} & {Rusanov} & {centred} \\ \hline
Lake at rest    & 2.6E-02 & 2.49E-07 & 3.6E-02 & 1.84E-07 & 3.1E-15 &
                                                                      1.22E-15 \\ \hline 
Isobaric        & 0.12  & 1.3E-08 & 9.1E-05 & 1.53E-09 & 0.17 & 1.81E-08 \\ \hline
Constant height & 0.02  & 2.6E-06 & 3.5E-04 & 6.0E-08 & 6.7E-07  & 2.6E-12 \\ \hline
\end{tabular}
\caption{$L^1$-errors in $h$, $u$ and $\theta$ for the Rusanov and the
  stabilised centred schemes with respect to the steady states.} 
\label{tab:l1_ss}
\end{table}

\subsection{Perturbation of a nonlinear steady state}
\label{subsec:pert_non-linear_ss_test}
The aim of this case study is to test the proposed scheme's ability to
resolve small perturbations of a nonlinear steady state. As the
scheme is energy stable, we expect that these perturbations remain
bounded in time and diminish as time increases. We start with the
following steady state solution from \cite{DZB+16}:  
\begin{equation}
  (h_s,u_s,\theta_s)^T(x)=(\exp(x),0,\exp(2x))^T,
\end{equation}
with the bottom topography $b(x)=6-2\exp(x)$. A small perturbation is
introduced into the steady state by setting the initial data as  
\begin{equation}
  (h, u, \theta)^{T}(0, x) = (h_s, u_s, \theta_s)^{T}(x) + (0.1, 0,
  0)^T\chi_{[-0.1, 0]}(x), 
\end{equation}
where $\chi_{[-0.1, 0]}$ is the characteristic function of the
sub-interval $[-0.1,0]$. 

\begin{figure}[htbp]
  \centering
  \includegraphics[height=0.25\textheight]{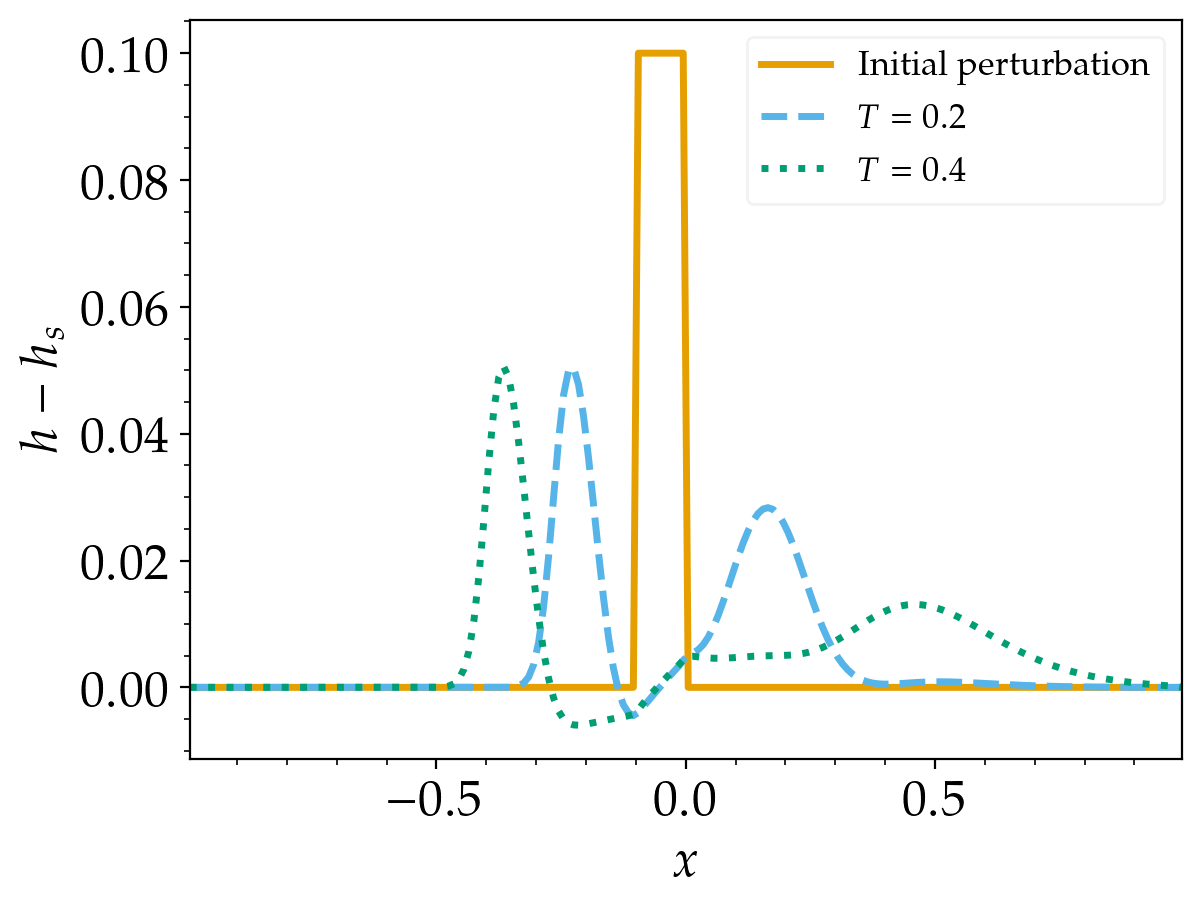}
    \includegraphics[height=0.25\textheight]{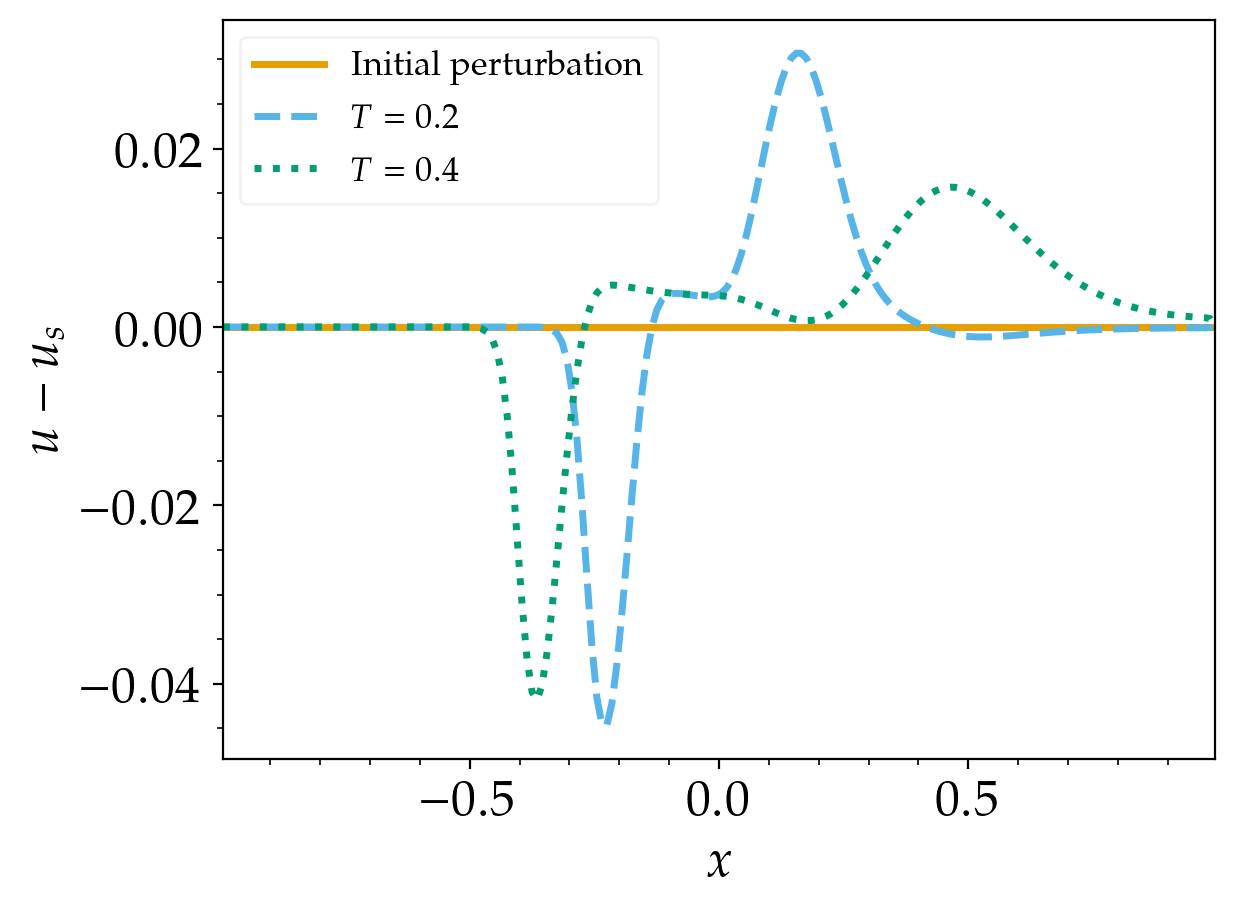}
    \caption{Perturbation in the water height and the velocity.}  
    \label{fig:nonlin_ss}
\end{figure}
The computational domain $[-1,1]$ is discretised using $200$ grid
points. In Figure~\ref{fig:nonlin_ss}, we plot the perturbations in the water height and the velocity obtained at times $T=0.2$ and $T=0.4$. We observe that the initial perturbation splits into two waves and propagates in opposite directions. Note that there are no spurious oscillations in the solutions and the introduced perturbations decrease with time. This corroborates the scheme's ability to maintain the well-balancing property by resolving the perturbations of a nonlinear steady state.

\subsection{1D Dam-break over flat bottom}
\label{subsec:1D_dam-break_flat}
We simulate a $1d$ dam break problem with flat bottom topography 
$(b\equiv 0)$ by considering the following Riemann initial data:
\begin{equation}
  (h, u, \theta)^T
  (0, x) =
  \begin{cases}
    (5, 0, 3)^T, & \text{if } x < 0, \\
    (1, 0, 5)^T, & \text{if } x > 0.
  \end{cases}
\end{equation}
The goal of this problem is to compare the performances of the centred
and the upwind schemes in the presence of discontinuities. We
discretise the computational domain $[-1,1]$ using $200$ grid 
points. In Figure~\ref{fig:1D_dam-break_flat}, we plot the water
height, the velocity, the temperature and the pressure profiles at
$T=0.2$. The results are compared with a reference solution 
obtained using an explicit Rusanov scheme on a fine mesh of $5000$
points. From Figure~\ref{fig:1D_dam-break_flat}, it is clear that
both the variants of the scheme are capable for capturing the rarefactions and 
the shocks in the solution. However, the results obtained using the
centred scheme leads to oscillations in the height and temperature
profiles. On the other hand, the upwind scheme leads to
non-oscillatory results. We also note that the present results are in
good agreement with those reported in \cite{CKL14, DZB+16}.  
\begin{figure}[htbp]
  \centering
  \includegraphics[height=0.25\textheight]{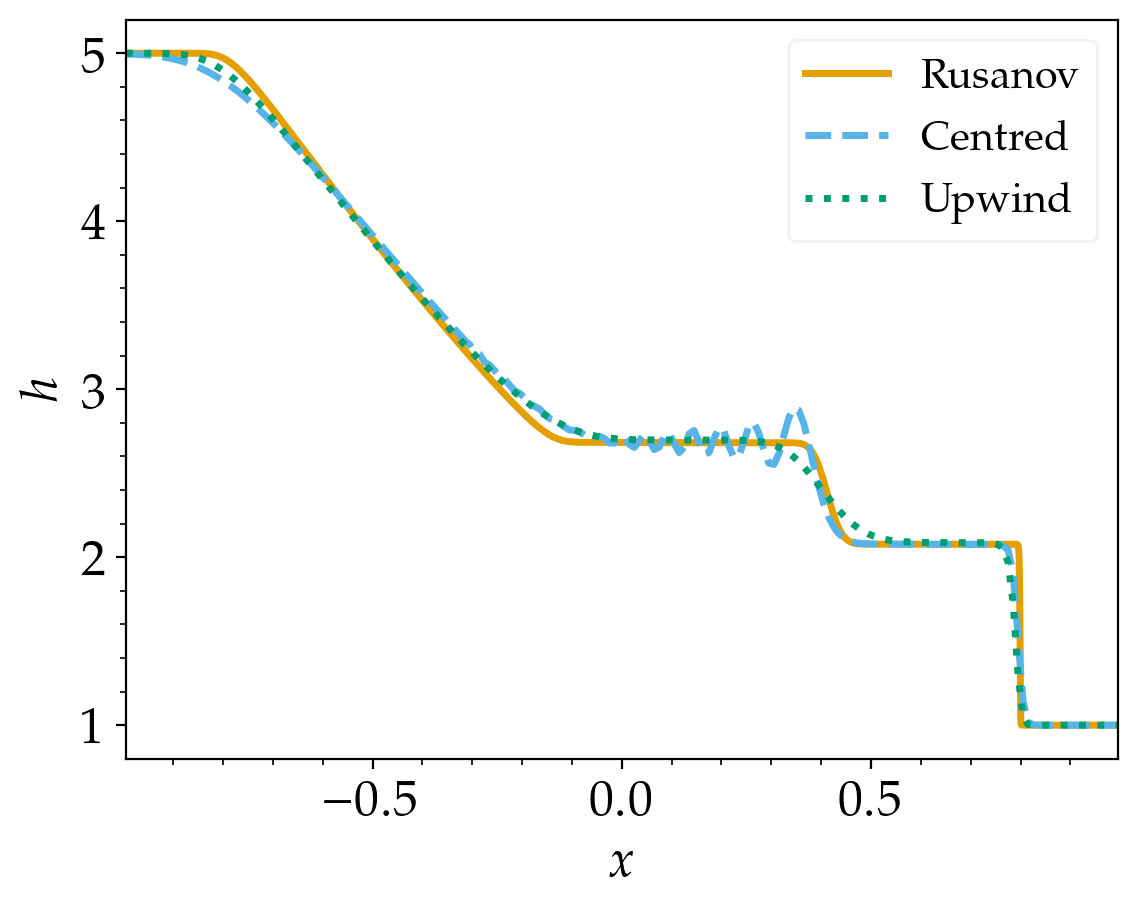}
  \includegraphics[height=0.25\textheight]{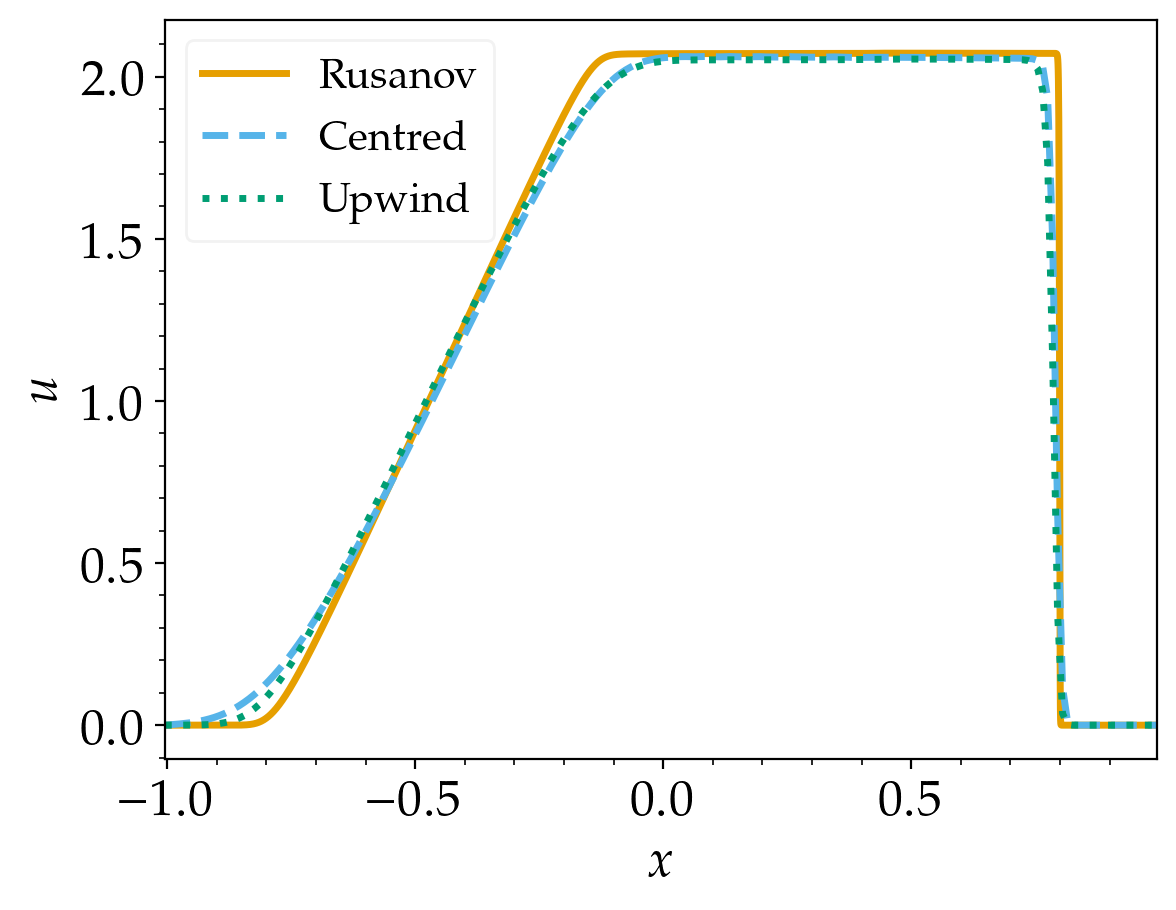}
  \includegraphics[height=0.25\textheight]{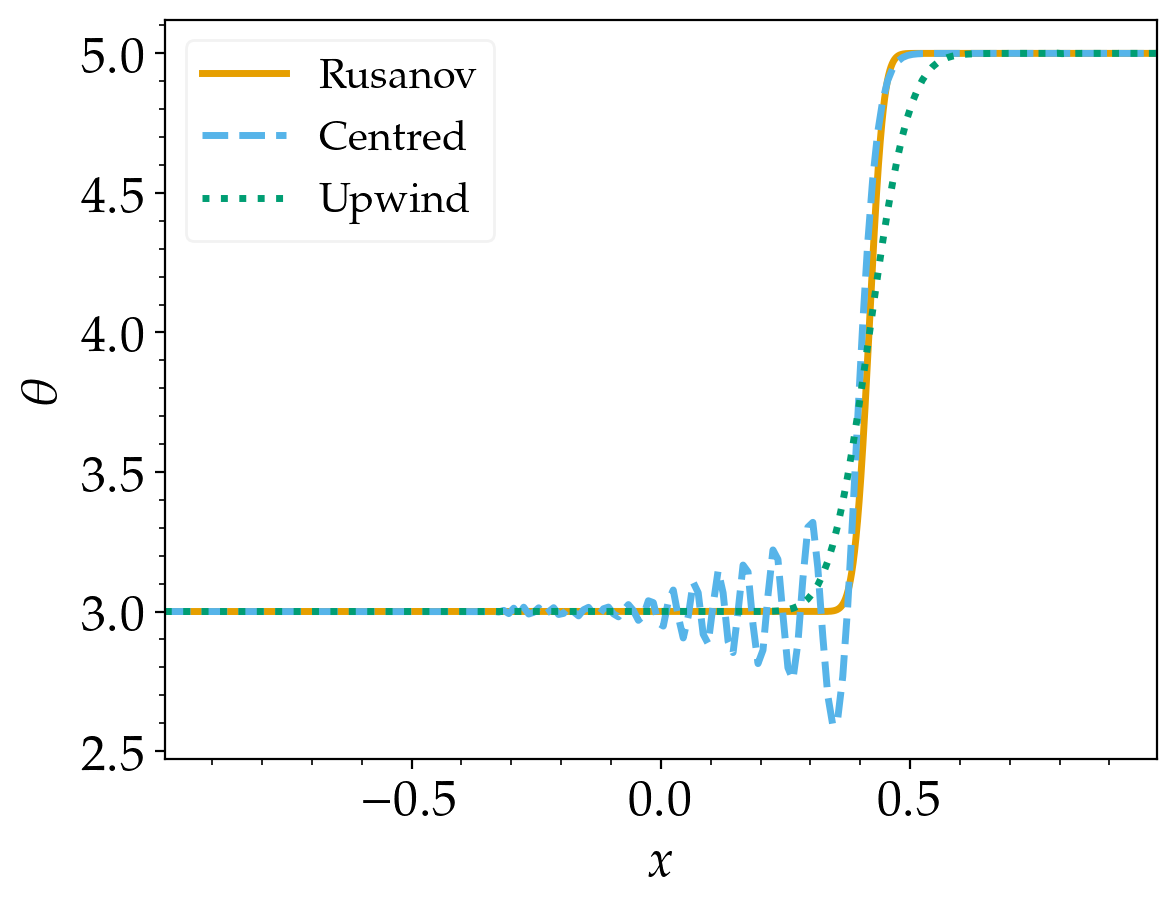}
  \includegraphics[height=0.25\textheight]{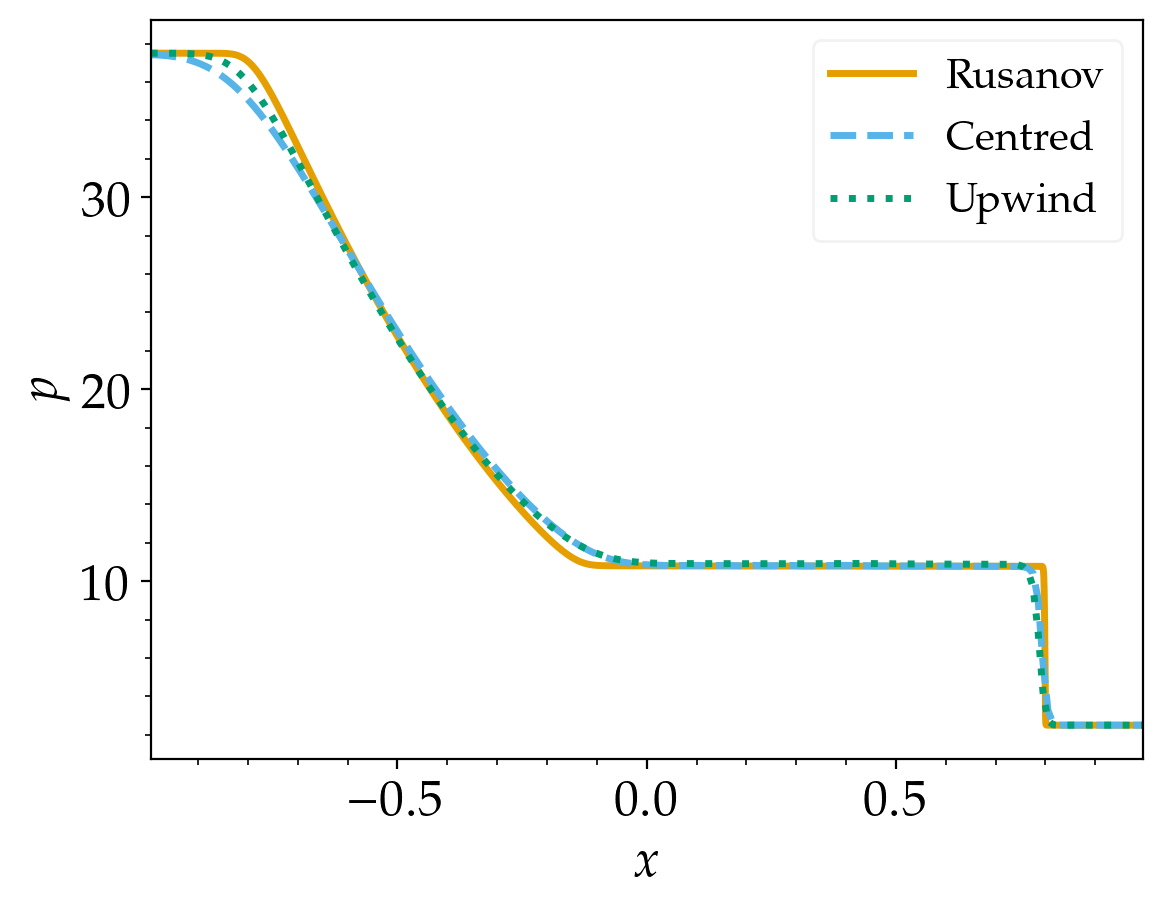}
  \caption{The water height, the velocity, the temperature and the
    pressure plots for the dam break test case with flat bottom
    topography at $T=0.2$.}   
  \label{fig:1D_dam-break_flat}
\end{figure}

\subsection{1D dam-break over non-flat bottom}
\label{subsec:1D_dam-break_nflat}
We consider a dam break test problem with a non-flat bottom topography
as given in \cite{CKL14}. The initial conditions read
\begin{equation}
(h, u, \theta)^T(0, x) =
\begin{cases}
(5-b(x),0,1)^T, & \text{if } x < 0, \\
(1-b(x), 0, 5)^T, & \text{if } x > 0,
\end{cases}
\end{equation}
with $b$ given as 
\begin{equation}
b(x)=\begin{cases} 
2(\cos(10\pi(x + 0.3)) + 1), & \text{if } -0.4 \leqs x \leqs -0.2, \\
0.5(\cos(10\pi(x - 0.3)) + 1), & \text{if } 0.2 \leqs x \leqs 0.4, \\
0, & \text{otherwise.}
\end{cases}
\end{equation}
The aim of this test problem is to assess the positivity preserving
property of the scheme. The computational domain $[-1,1]$ is divided
into $200$ grid points and the final time is set to $T=0.3$. Initially,
the area near $x=0.3$ is almost dry as the water height is
$h(0,0.3)=0$. In Figure~\ref{fig:1D_dam-break_nflat}, we present
the water height, the velocity, the temperature and the pressure
profiles against the results of an explicit Rusanov scheme on a
resolved mesh of $5000$ mesh points. As expected, the present scheme
can very well preserve the positivity of the water height and the
temperature in the almost dry region.
\begin{figure}[htbp]
  \centering
  \includegraphics[height=0.25\textheight]{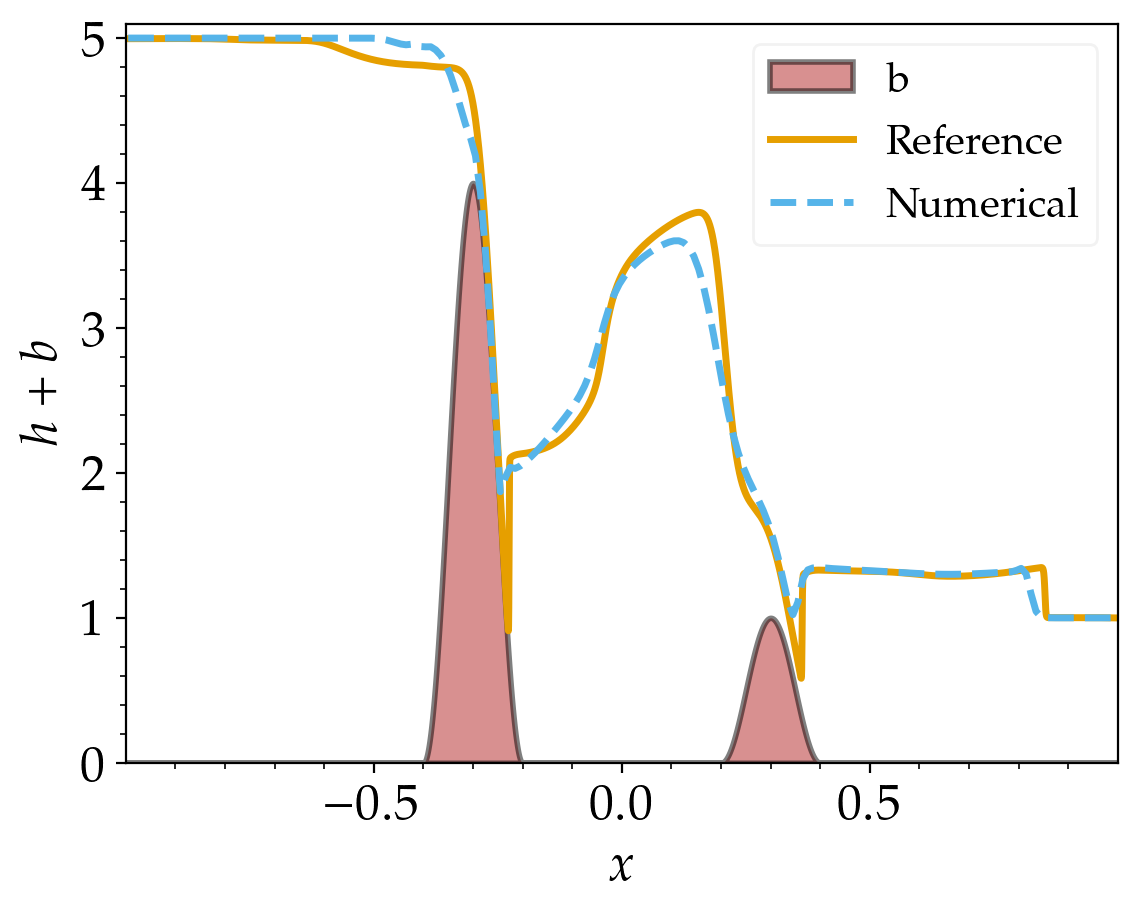}
  \includegraphics[height=0.25\textheight]{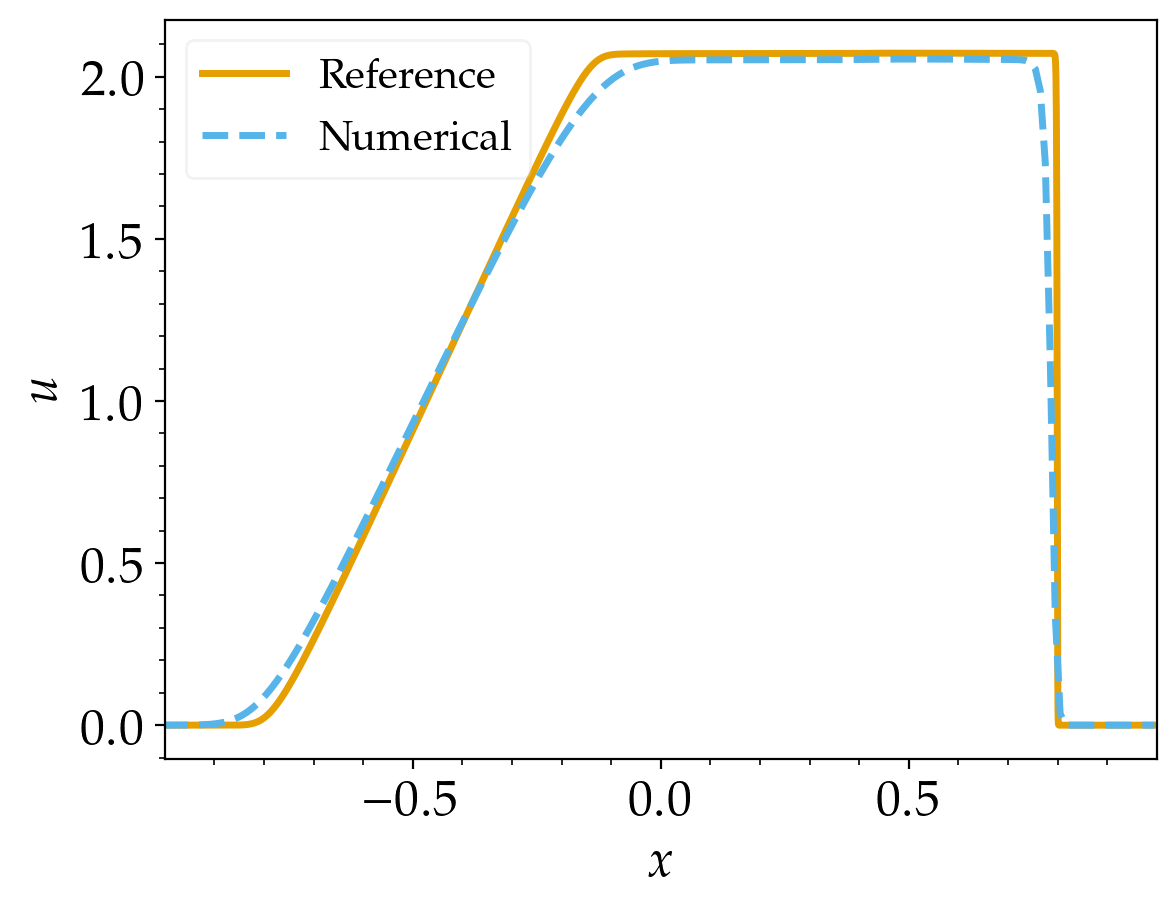}
  \includegraphics[height=0.25\textheight]{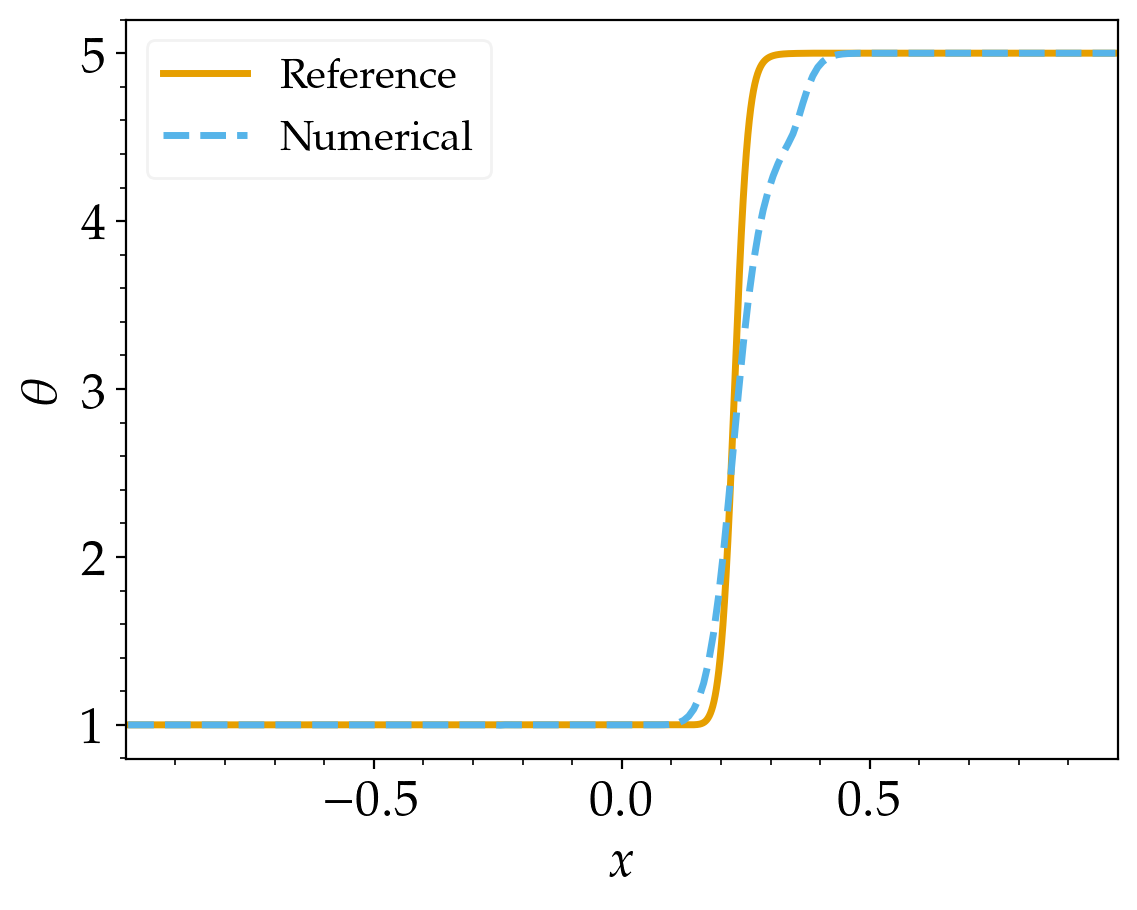}
  \includegraphics[height=0.25\textheight]{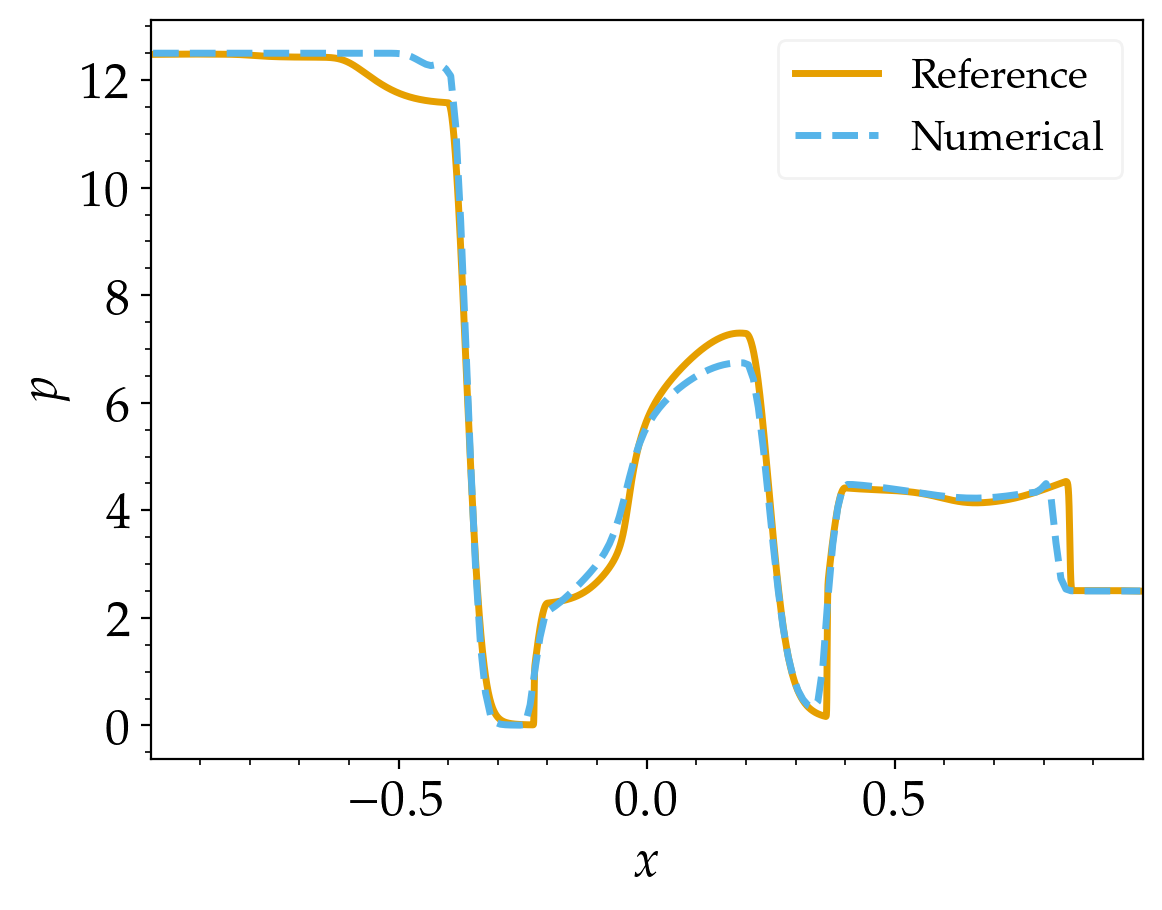}
  \caption{The water surface, the velocity, the temperature and
    pressure plots for the dam break test case with non-flat bottom
    topography at $T=0.3$.}   
  \label{fig:1D_dam-break_nflat}
\end{figure}

\subsection{2D steady state}
\label{subsec:2D_ss}
We now show the proposed centred scheme's ability to accurately
preserve $2d$ steady states through a numerical case
study. We consider a test problem from \cite{CKL14}
which consists of two lake at rest type steady states
\eqref{eq:lake_rest_ss} connected through a temperature jump
corresponding to the isobaric steady state \eqref{eq:isobaric_ss}. The
initial data read 
\begin{equation}
\label{eq:2D_ss}
(h, u, v, \theta)^T(0, x, y) =
\begin{cases}
(3, 0, 0, \frac{4}{3})^T, & \text{if } x^2+y^2 < 0.25, \\
(2, 0, 0, 3)^T, & \text{otherwise}.
\end{cases}
\end{equation}
The bottom topography is given by
\begin{equation}
    b(x, y) =
\begin{cases}
0.5 \exp\left(-100\left((x + 0.5)^2 + (y + 0.5)^2\right)\right), & x < 0, \\
0.6 \exp\left(-100\left((x - 0.5)^2 + (y - 0.5)^2\right)\right), & x > 0.
\end{cases}
\end{equation}
The computational domain $[-1,1]\times[-1,1]$ is divided uniformly by
a $100\times100$ mesh. We simulate the numerical experiment up to a
final time $T=0.12$ and plot the results in Figure~\eqref{fig:2D_ss}
for the proposed scheme as well as the Rusanov scheme. It is clear
from the figure that the proposed scheme can exactly preserves the
steady state, whereas the solution obtained using the Rusanov scheme
is way too smeared near discontinuities. Some wiggles can be observed in the pressure profile computed using the non well-balanced scheme, whereas the well-balanced scheme yields non-oscillatory results. Furthermore, the well-balanced scheme maintains the circular shape of the waves intact.     
\begin{figure}[htbp]
  \centering
  \includegraphics[height=0.21\textheight,
  width=0.49\textwidth]{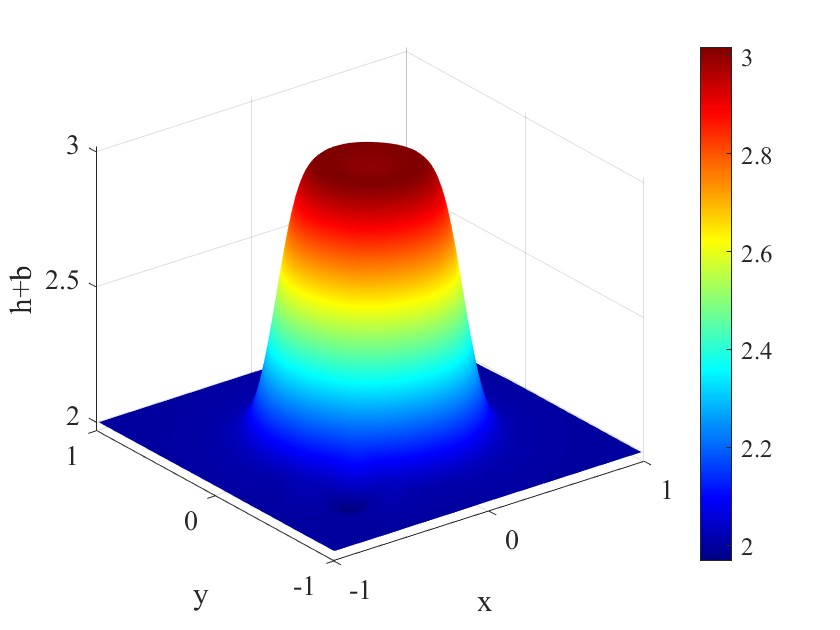} 
  \includegraphics[height=0.21\textheight,
  width=0.49\textwidth]{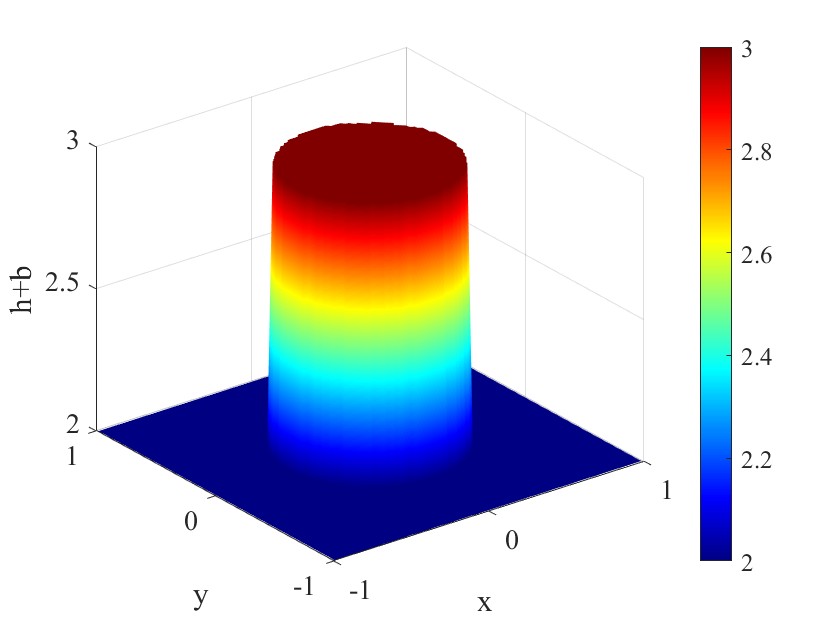} 
  \includegraphics[height=0.21\textheight,
  width=0.49\textwidth]{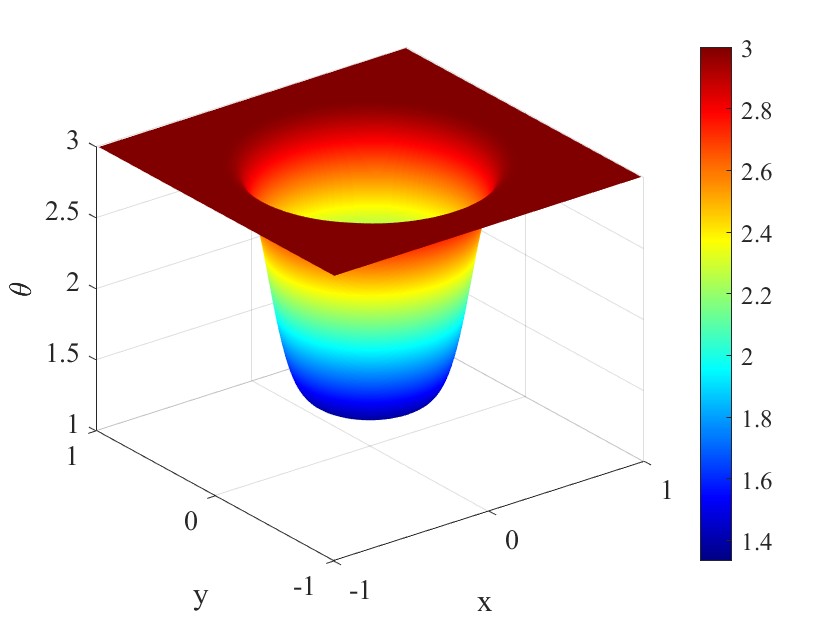} 
  \includegraphics[height=0.21\textheight,
  width=0.49\textwidth]{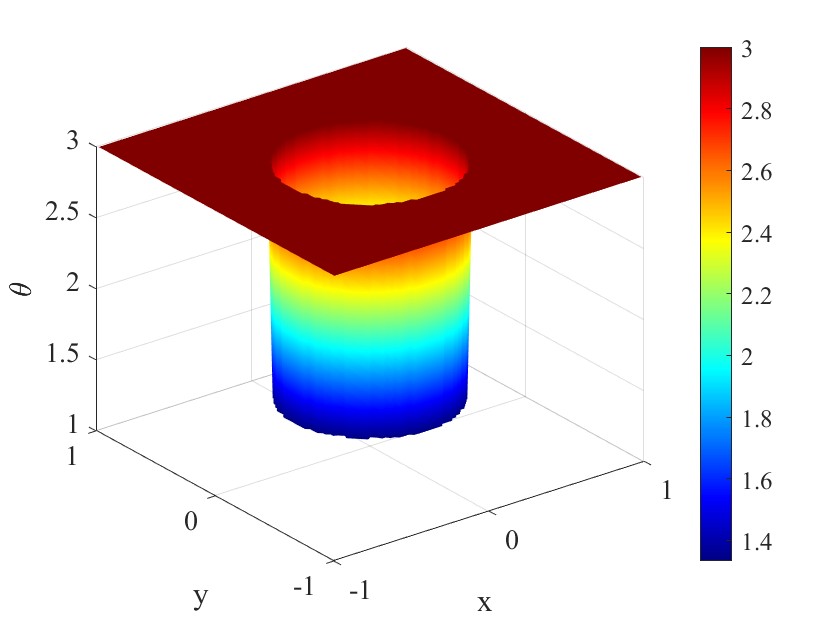} 
  \includegraphics[height=0.21\textheight,
  width=0.49\textwidth]{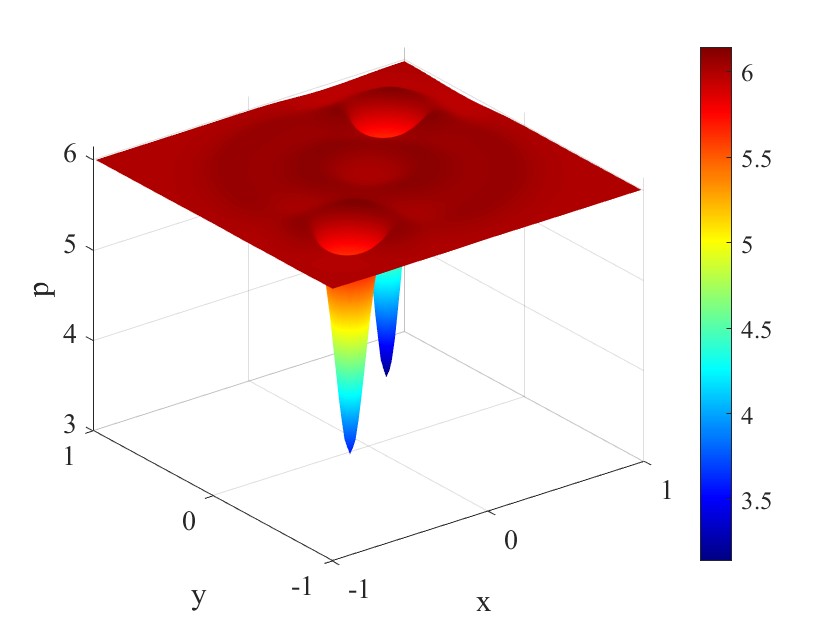} 
  \includegraphics[height=0.21\textheight,
  width=0.49\textwidth]{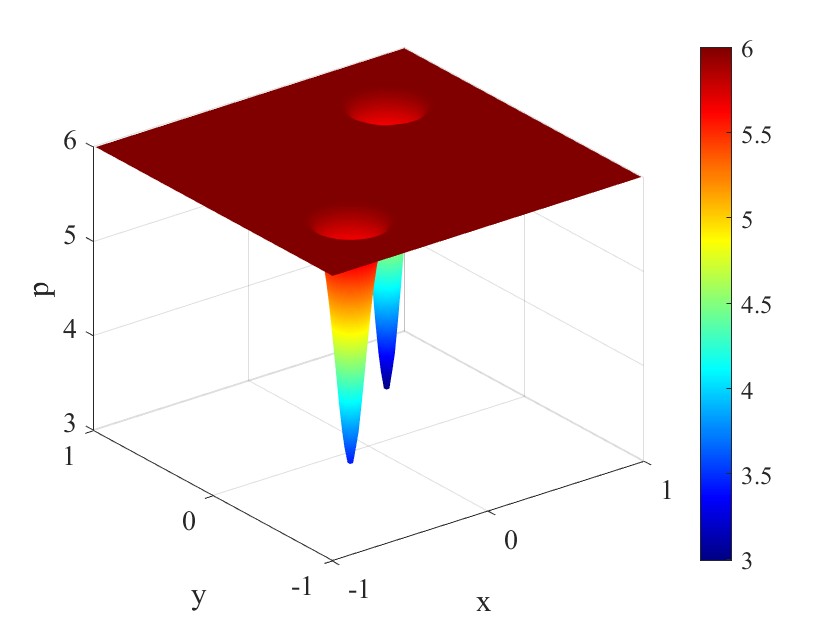} 
  \caption{The water surface, the temperature and the pressure plots
    for the $2d$ steady state problem calculated using the non
    well-balanced scheme (left) and the well-balanced scheme (right)
    at $T=0.12$.}  
  \label{fig:2D_ss}
\end{figure}

\subsection{2D perturbation of steady state}
\label{subsec:2D_pert_ss}
In this case study, we use the same setting as in the previous case
study but the initial condition \eqref{eq:2D_ss} is perturbed within a
small annulus near the centre of the computational domain. We consider the initial data  
\begin{equation}
(h, u, v, \theta)^T(0, x, y) =
\begin{cases}
(3+0.1, 0, 0, \frac{4}{3})^T, & \text{if } 0.01 < x^2+y^2 < 0.09, \\
(3, 0, 0, \frac{4}{3})^T, & \text{if } 0.09 < x^2+y^2 < 0.25 \text{ or } x^2+y^2 < 0.01, \\
(2, 0, 0, 3)^T, & \text{otherwise}.
\end{cases}
\end{equation}
We simulate the experiment up to a final time $T=0.15$ and the results
are shown in Figure~\ref{fig:2D_pert_ss}. As in the previous test
case, the well-balanced scheme gives a sharper resolution of the
solution compared to its non well-balanced counterpart. The
well-balanced solution clearly shows the propagating circular pressure waves and
their interaction with the perturbations. On the other hand, the Rusanov scheme smears out the flow features and it cannot capture the complex structure of the solution. 
\begin{figure}[htbp]
  \centering
  \includegraphics[height=0.21\textheight,
  width=0.49\textwidth]{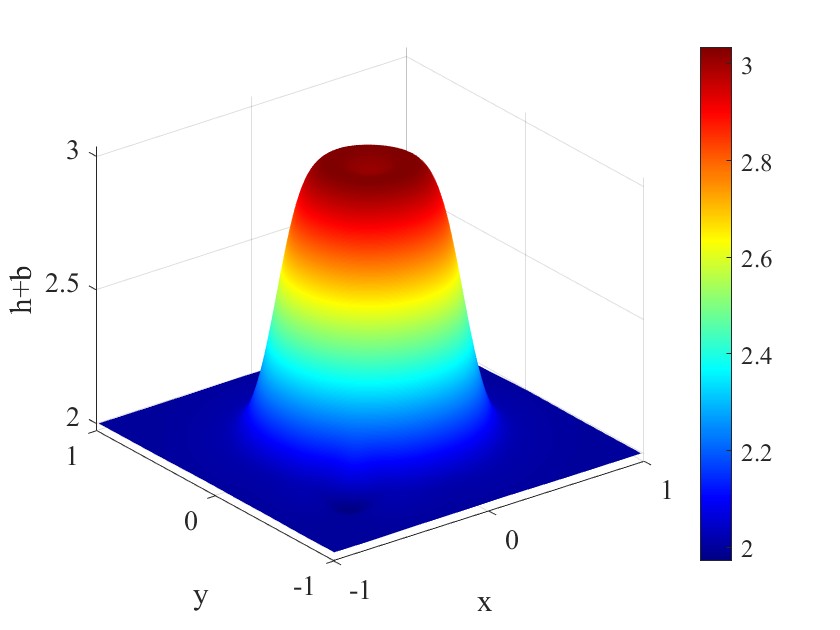} 
  \includegraphics[height=0.21\textheight,
  width=0.49\textwidth]{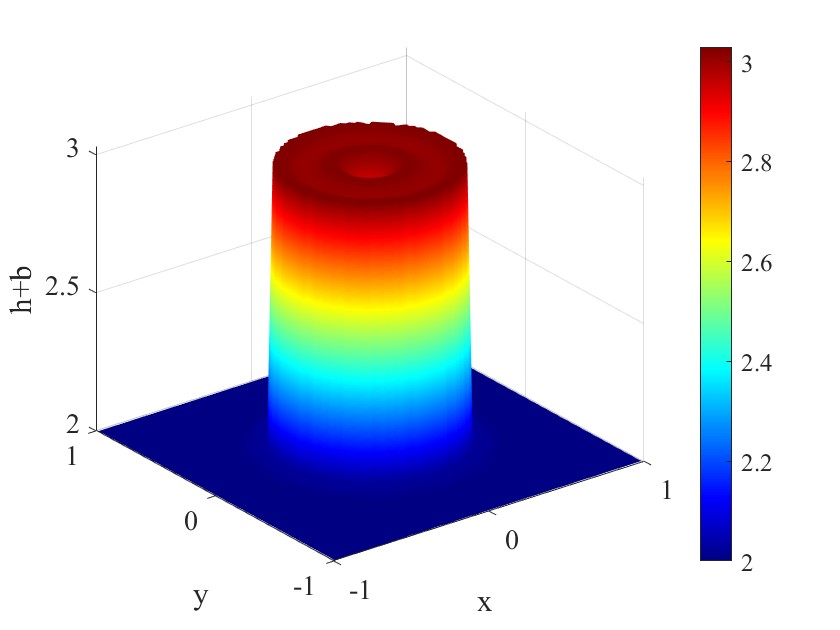} 
  \includegraphics[height=0.21\textheight,
  width=0.49\textwidth]{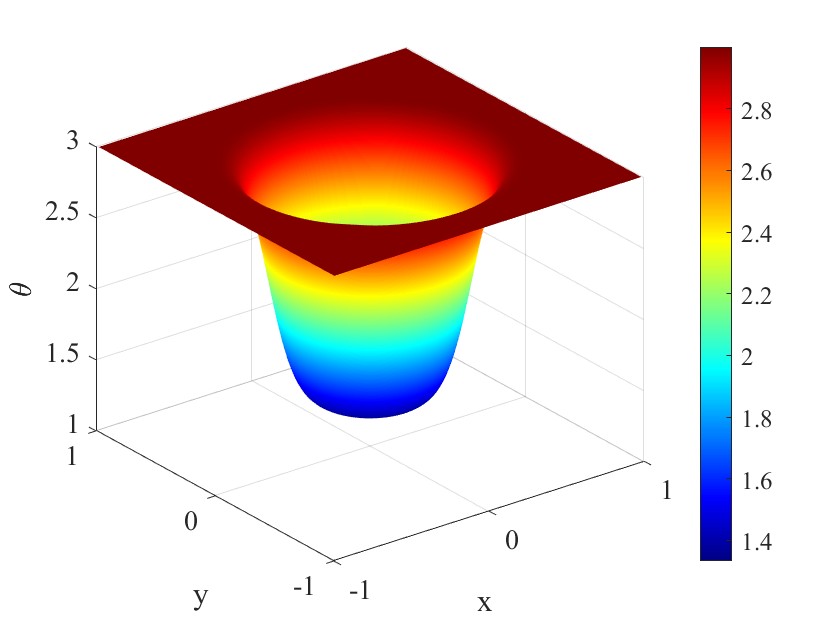} 
  \includegraphics[height=0.21\textheight,
  width=0.49\textwidth]{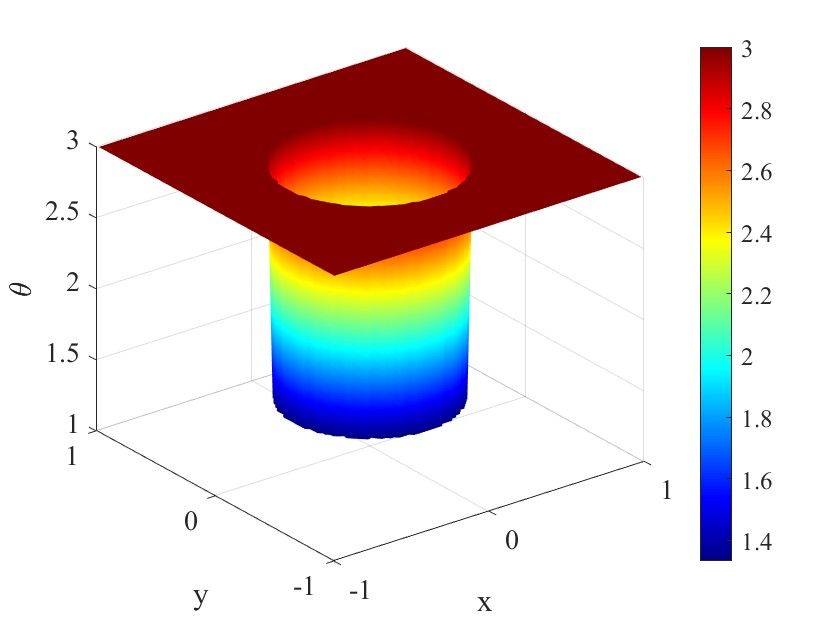} 
  \includegraphics[height=0.21\textheight,
  width=0.49\textwidth]{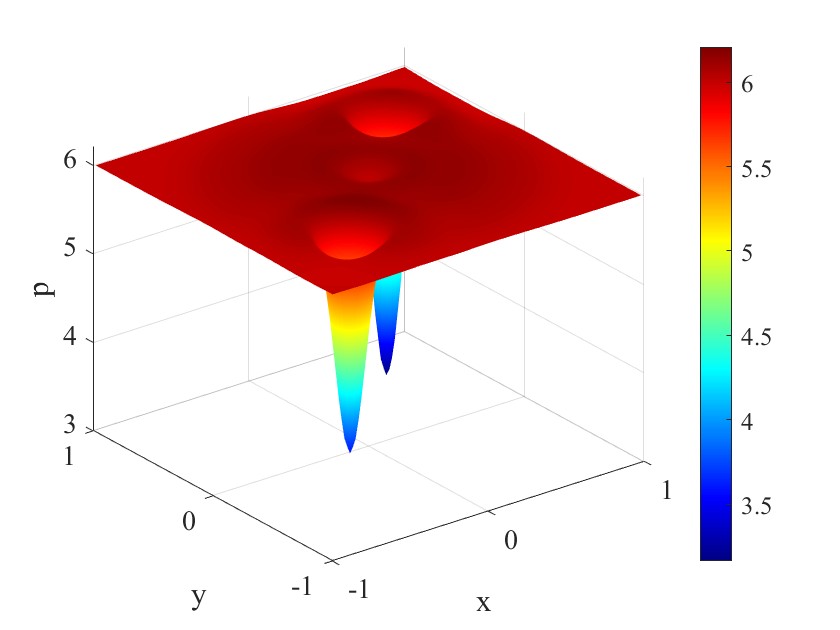} 
  \includegraphics[height=0.21\textheight,
  width=0.49\textwidth]{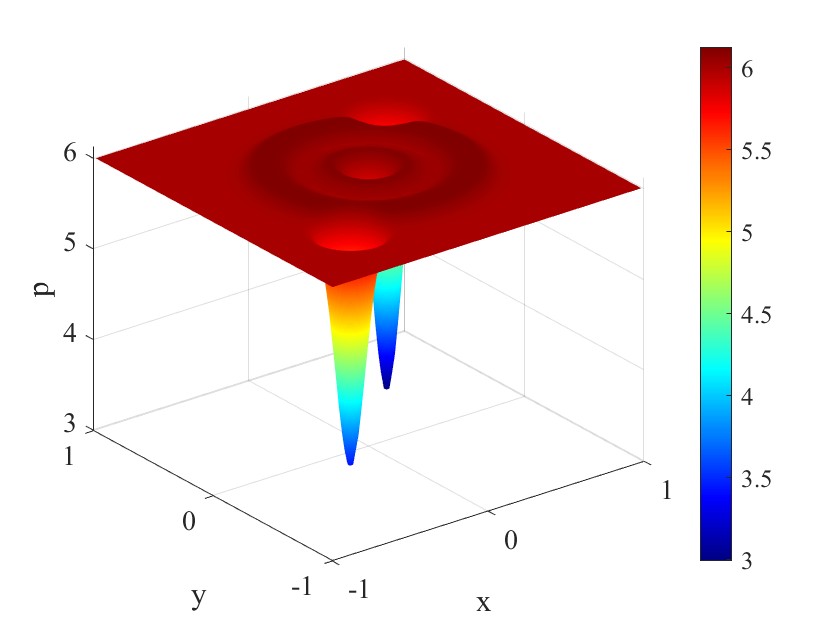} 
  \caption{The water surface, the temperature and the pressure plots
    for the $2d$ perturbed steady state calculated using the non
    well-balanced scheme (left) and the well-balanced scheme (right)
    at $T=0.15$.} 
  \label{fig:2D_pert_ss}
\end{figure}

\subsection{2D circular dam-break problem}
\label{subsec:2D_rad_dam-break}
We consider a classical $2d$ dam break problem from \cite{CKL14} as
given below  
\begin{equation}
(h, u, v, \theta)^T(0, x, y) =
\begin{cases}
(2, 0, 0, 1)^T, & \text{if } x^2+y^2 < 0.25, \\
(1, 0, 0, 1.5)^T, & \text{otherwise},
\end{cases}
\end{equation}
with flat bottom topography ($b\equiv 0$). We simulate the experiment
on the computational domain $[-1,1]\times[-1,1]$ and up to a final time
$T=0.15$. In Figure~\ref{fig:2D_rad_dam-break}, we depict the
results obtained using an explicit Rusanov scheme and the proposed
scheme on a mesh of $200 \times 200$ grid points. The solution
consists of an outward moving shock, inward moving rarefaction and a
contact in between. It is clear from the figure that the explicit
Rusanov scheme yields overly smeared profiles near the contact region,
whereas the proposed explicit scheme sharply resolves the contact
surface.  
\begin{figure}[htbp]
  \centering
  \includegraphics[height=0.21\textheight,
  width=0.49\textwidth]{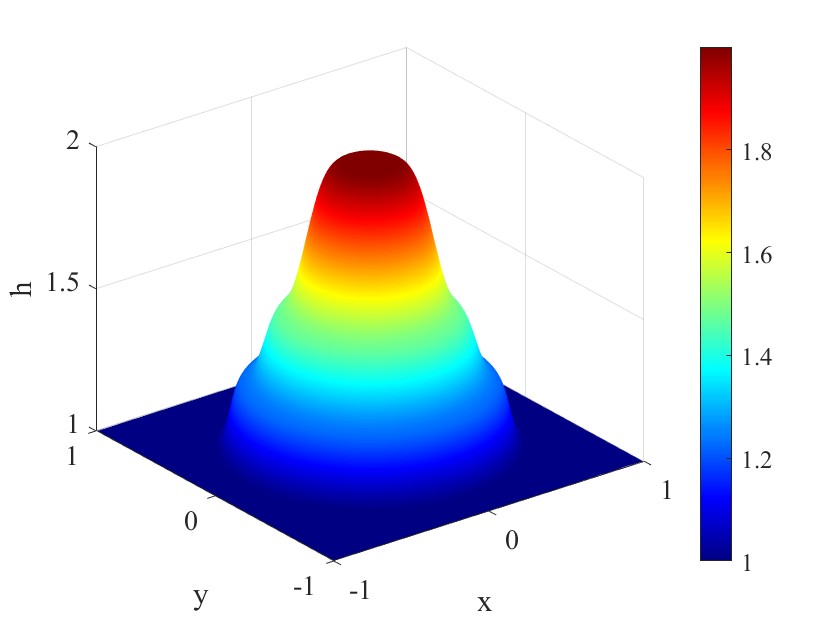} 
  \includegraphics[height=0.21\textheight,
  width=0.49\textwidth]{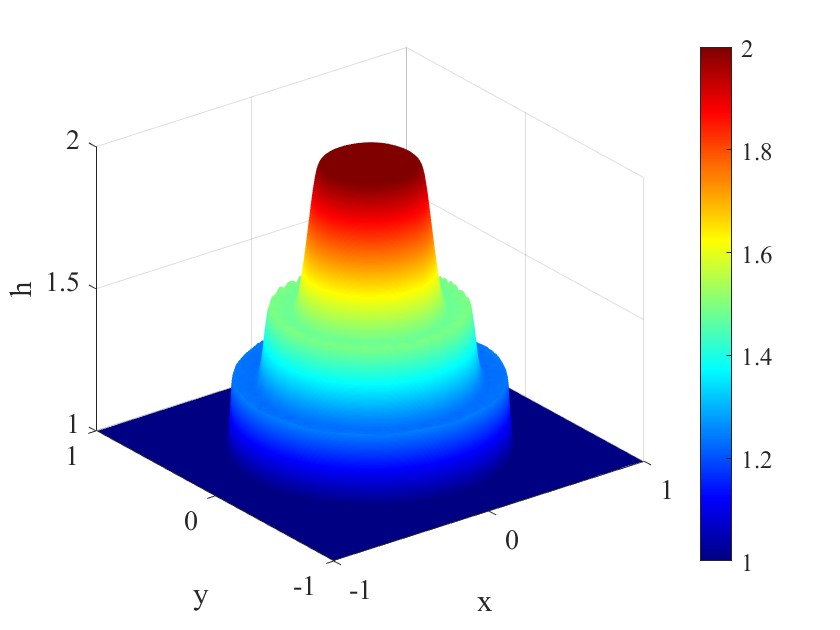} 
  \includegraphics[height=0.21\textheight,
  width=0.49\textwidth]{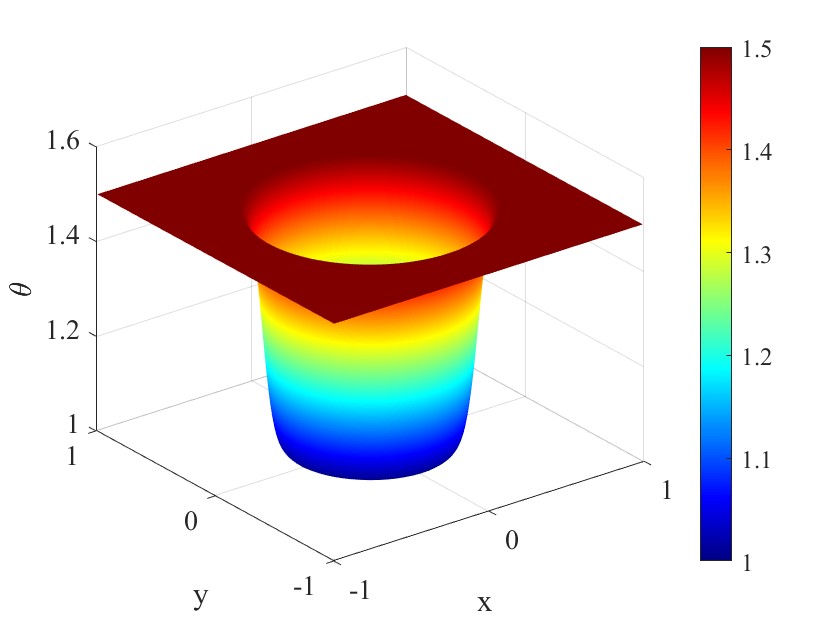} 
  \includegraphics[height=0.21\textheight,
  width=0.49\textwidth]{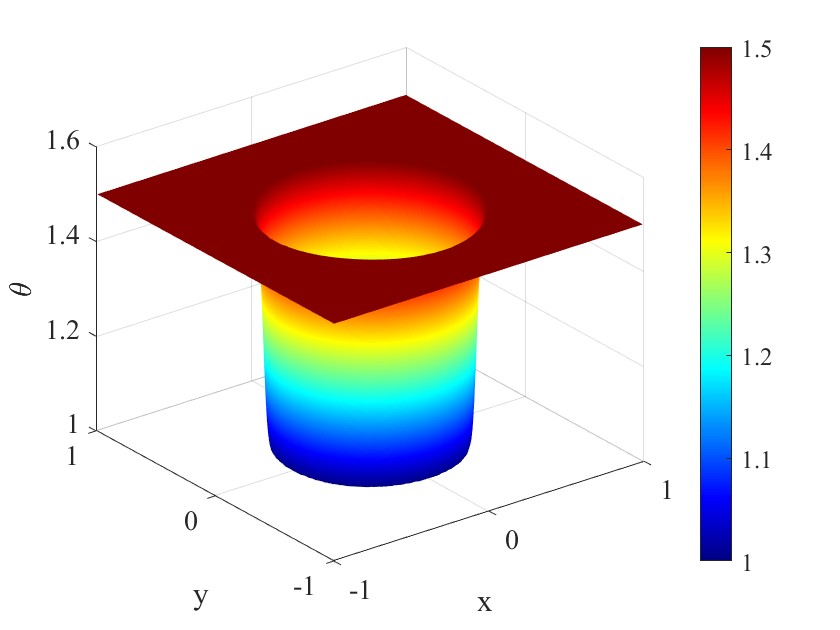} 
  \includegraphics[height=0.21\textheight,
  width=0.49\textwidth]{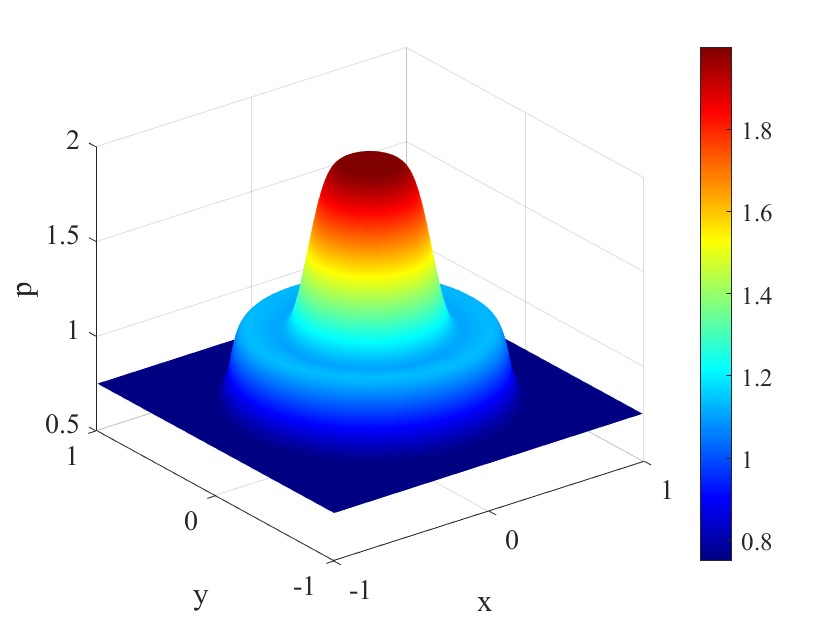} 
  \includegraphics[height=0.21\textheight,
  width=0.49\textwidth]{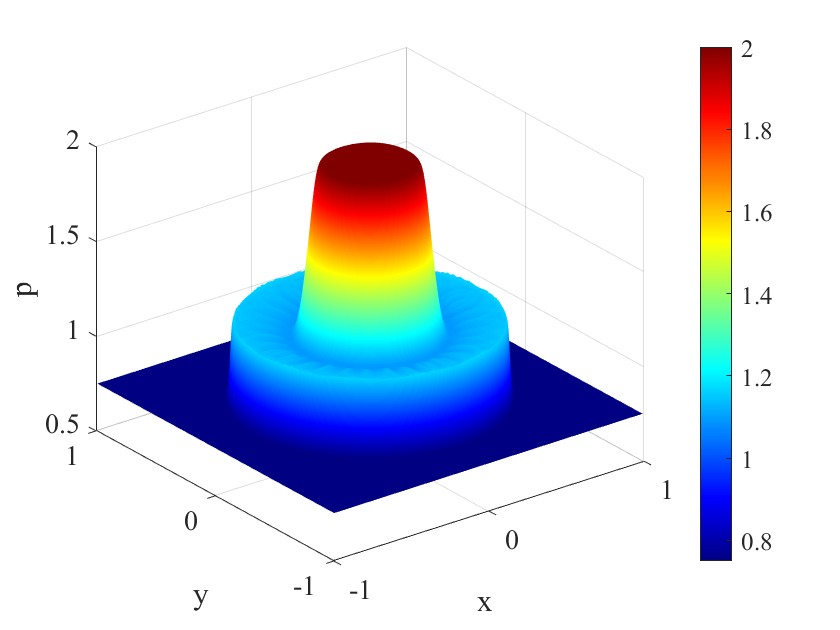} 
  \caption{The water surface, temperature and pressure plots for the
    $2d$ circular dam break test case calculated using the non
    well-balanced scheme (left) and the well-balanced scheme (right) at
    $T=0.15$.} 
  \label{fig:2D_rad_dam-break}
\end{figure}

\subsection{2D rectangular dam-break problem.}
\label{subsec:2D_rect_dam-break}
This test problem is borrowed from \cite{TK15}. The initial data read
\begin{equation}
  (h, u, v, \theta)^T(0, x, y) =
  \begin{cases}
    (2, 0, 0, 1)^T, & \text{if } \abs{x} \leqs 0.5, \\
    (1, 0, 0, 1.5)^T, & \text{otherwise},
  \end{cases}
\end{equation}
with the bottom topography $b$ is taken as $b\equiv 0$.
The computational domain $[-1,1]\times[-1,1]$ is divided uniformly by
a $200\times200$ mesh. In Figure~\ref{fig:2D_rect_dam-break}, we
present the $3d$ plots of the water height, the velocity, the
temperature and the velocity at $T=0.2$ obtained using the proposed
scheme. 
\begin{figure}[htbp]
  \centering
  \includegraphics[height=0.21\textheight,
  width=0.49\textwidth]{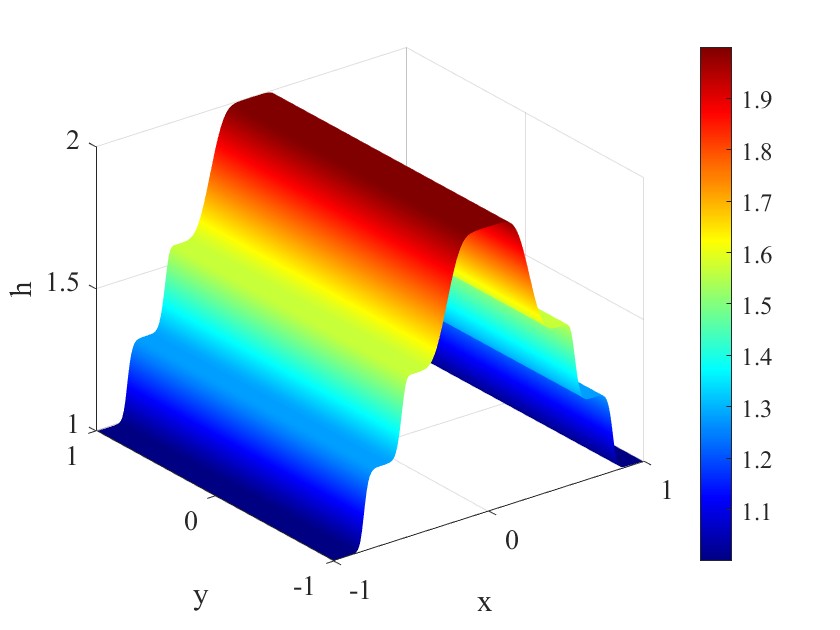} 
  \includegraphics[height=0.21\textheight,
  width=0.49\textwidth]{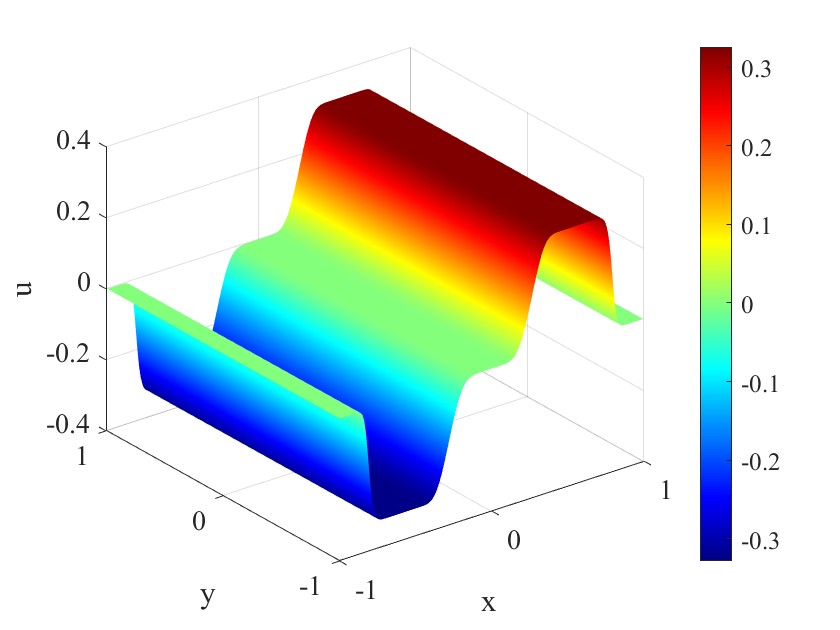} 
  \includegraphics[height=0.21\textheight,
  width=0.49\textwidth]{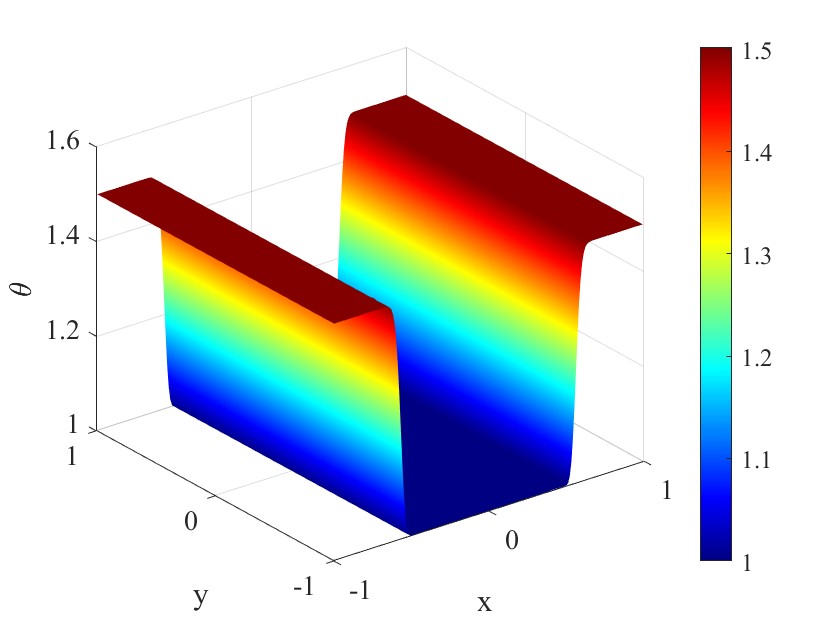} 
  \includegraphics[height=0.21\textheight,
  width=0.49\textwidth]{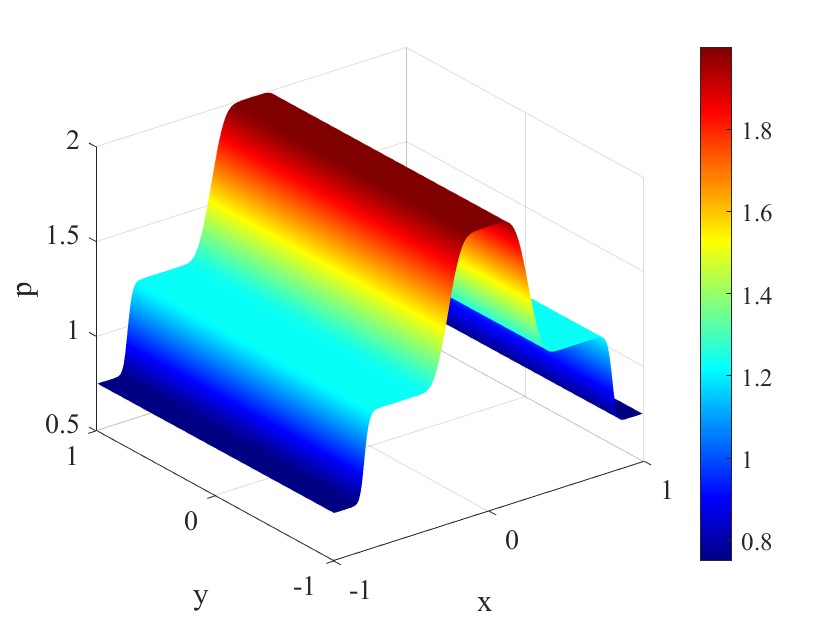} 
  \caption{The water surface, the velocity, the temperature and the
    pressure plots for the $2d$ rectangular dam break test case at
    $T=0.2$.} 
  \label{fig:2D_rect_dam-break}
\end{figure}
Additionally, Figure~\ref{fig:2D_rect_dam-break_crsn} depicts a
comparison of the cross-sectional results along the $x$-axis against a
reference solution obtained using the Rusanov scheme, treated
as a quasi-$1d$ problem on $5000$ grid points. The comparison clearly
reveals the good agreement between the results obtained on different
mesh resolutions and the chosen reference solution. Furthermore, both
the figures demonstrates the present scheme's ability to
accurately capture the two shock waves, the two contact waves and the
two rarefaction waves in the water height. 
\begin{figure}[htbp]
  \centering
  \includegraphics[height=0.25\textheight]{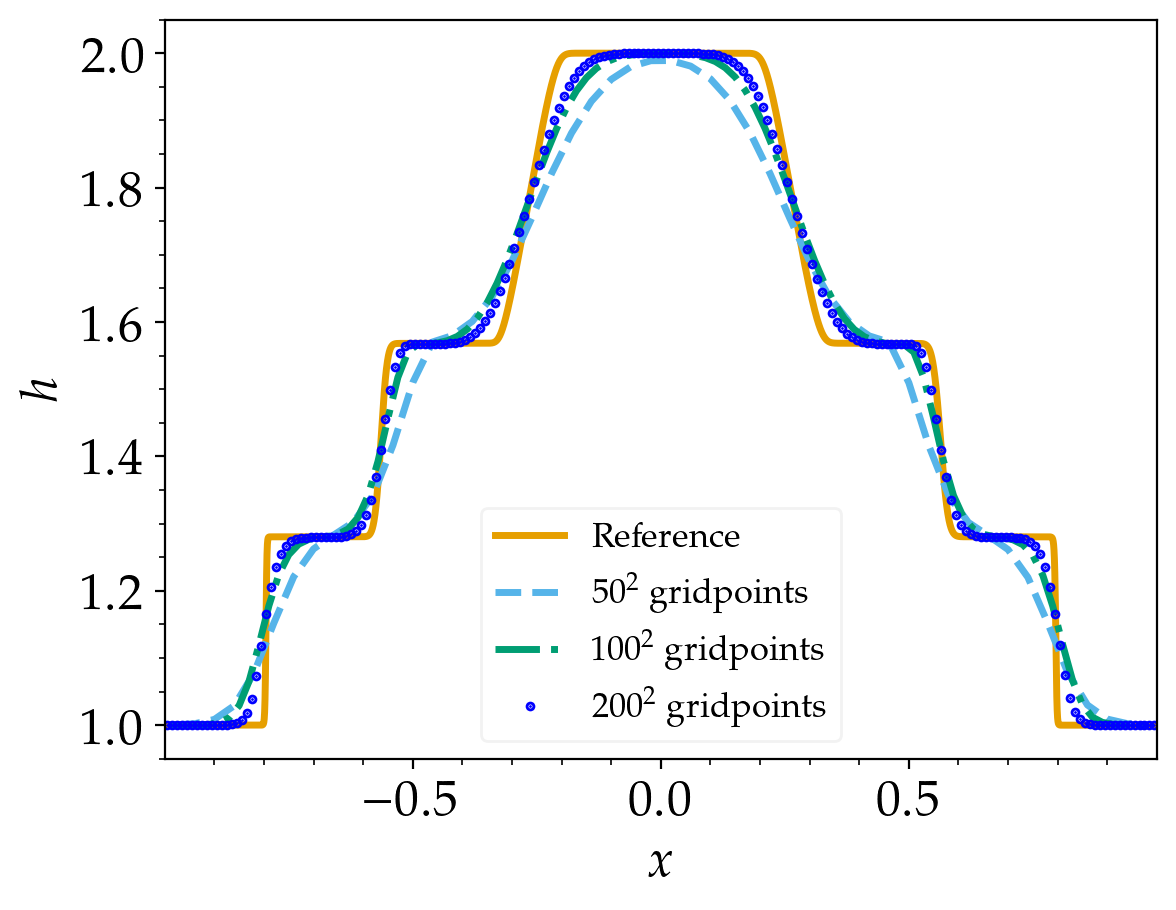}   
  \includegraphics[height=0.25\textheight]{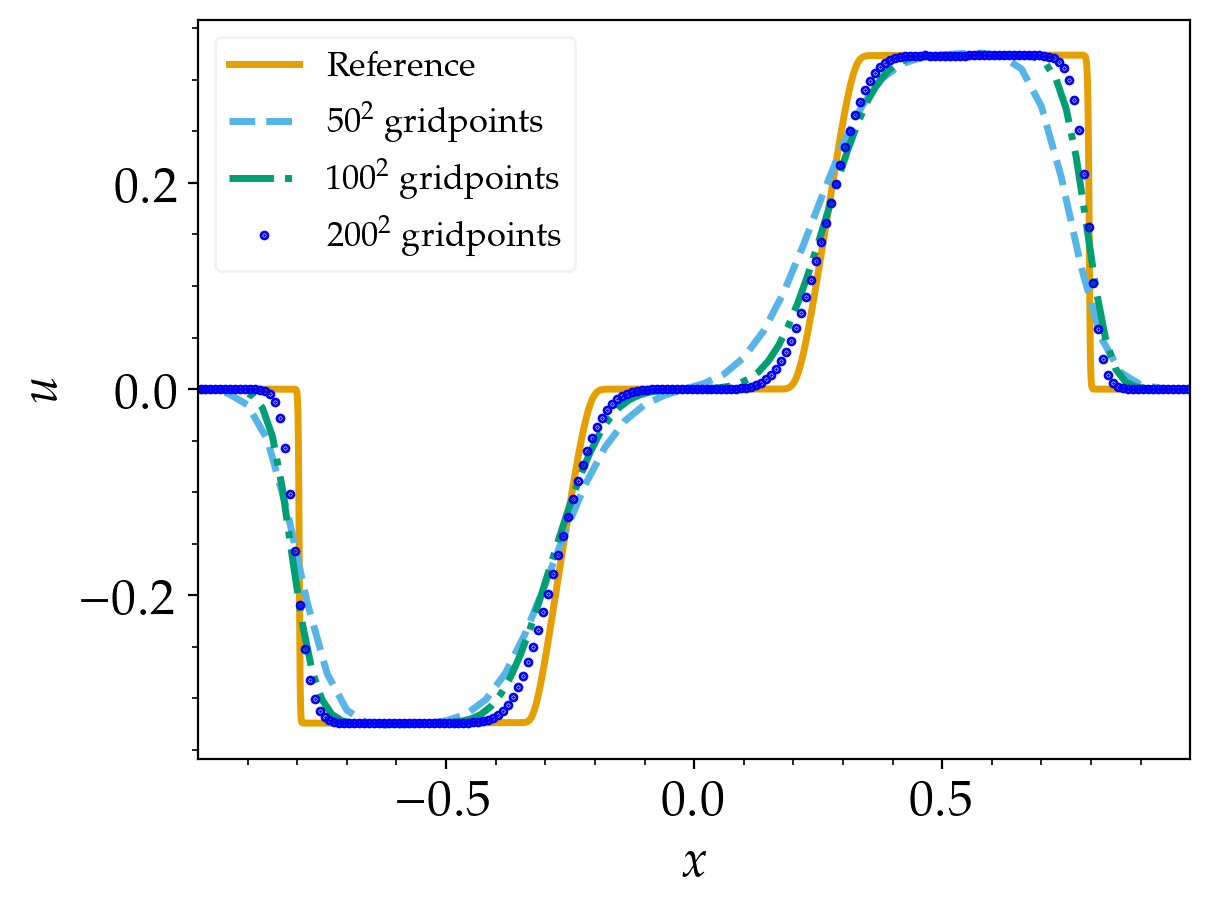} 
  \includegraphics[height=0.25\textheight]{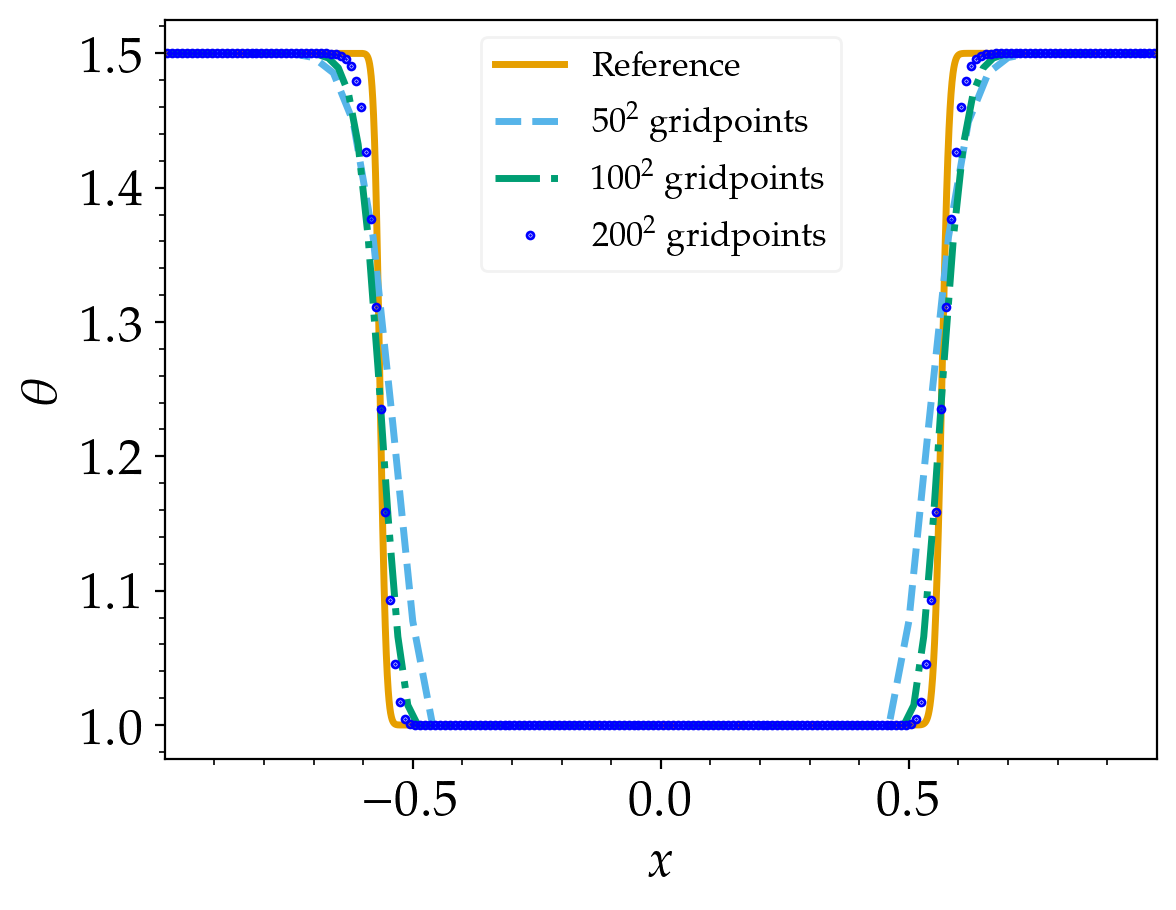} 
  \includegraphics[height=0.25\textheight]{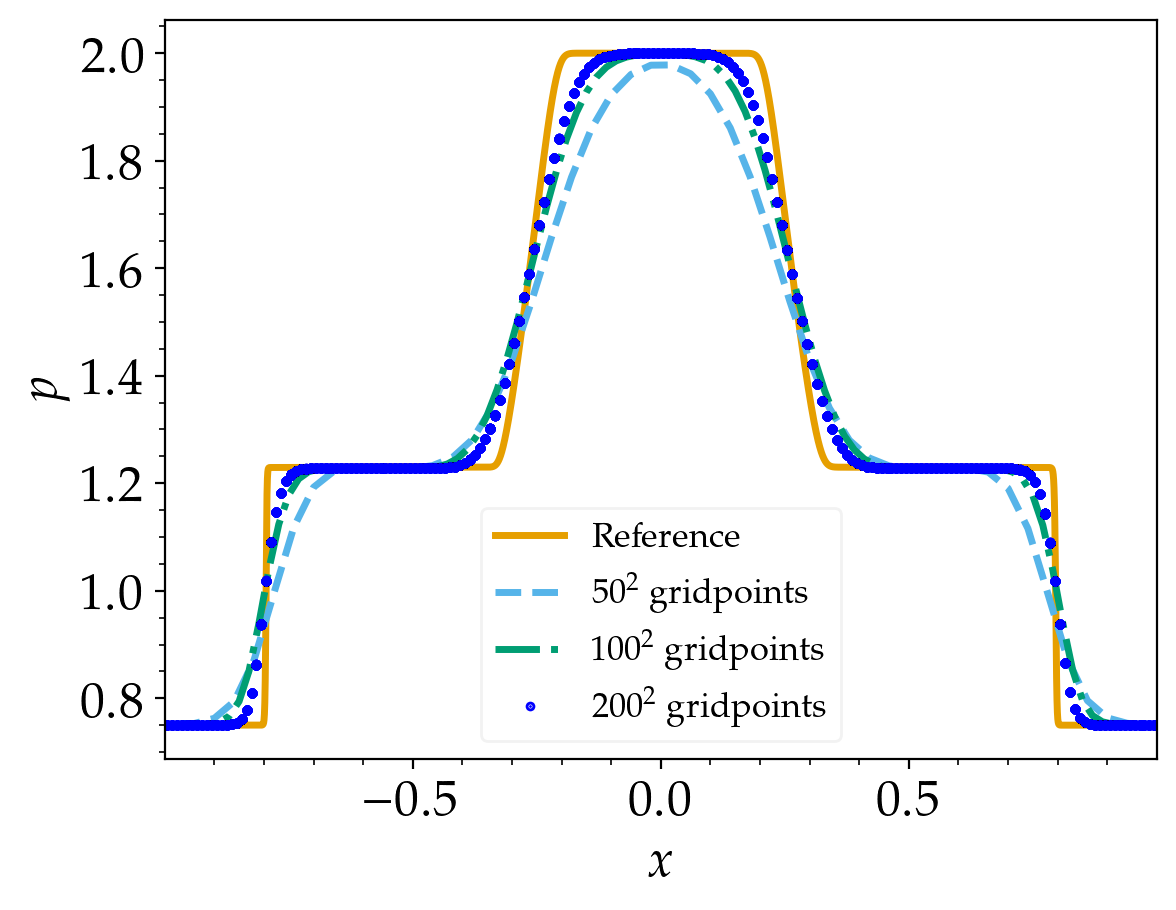}
  \caption{Cross sections of the water height, the velocity, the
    temperature and the pressure for the rectangular dam break test
    problem at $T=0.2$ for different mesh resolutions.}    
  \label{fig:2D_rect_dam-break_crsn}
\end{figure}

\section{Concluding Remarks}
\label{sec:conclusion}
We have constructed an energy stable, structure preserving and
well-balanced scheme for the Ripa system of shallow water
equations. The key to energy stability is the introduction of
appropriate shift terms in the convective fluxes of mass and momenta, 
the pressure gradient and the source term. The fully explicit in time
and finite volume in space scheme preserves the positivity of the
water height and the temperature and is weakly consistent with the
continuous model equations. The centred variant of the scheme is
energy stable under suitable restrictions on the timestep and the stabilisation
parameters, whereas the upwind variant is energy consistent upon mesh
refinement. The results from several benchmark case studies confirm
the theoretical findings. The simulations involving stationary initial
data clearly shows that the present scheme preserves the three
hydrostatic steady states under consideration very well, and that it can
accurately resolve perturbations which arise from small perturbations of steady
states. We have also demonstrated via case studies the scheme’s
desired capabilities, such as maintaining the positivity and capturing
the discontinuities in solutions.

\bibliographystyle{abbrv}
\bibliography{references}
\end{document}